\documentclass[smallextended]{svjour3}       
\smartqed  
\usepackage{amsmath, amssymb}

 \usepackage{graphicx}

\newtheorem{prop}{Proposition}
\newtheorem{rem}{Remark}

\def\C#1{{\mathcal {#1}}}

\begin{document}

\title{Semidiscrete variable time-step $\theta$-scheme for nonmonotone evolution inclusion\thanks{The research was supported by the Marie Curie International Research Staff Exchange Scheme Fellowship within the 7th European Community Framework Programme under Grant Agreement No. 295118,
by the National Science Center of Poland under grant no. N N201 604640, and by the International Project co-financed by the Ministry of Science and Higher Education of Republic of Poland under grant no. W111/7.PR/2012}
}

\titlerunning{Semidiscrete $\theta$-scheme for nonmonotone evolution inclusion}        

\author{Piotr Kalita}

\institute{P. Kalita \at
              Institute of Computer Science, Faculty of Mathematics and Computer Science, Jagiellonian University, ul. \L{}ojasiewicza 6, 30348 Krakow, Poland\\
              \email{piotr.kalita@ii.uj.edu.pl}  
}

\date{Received: date / Accepted: date}

\maketitle

\begin{abstract}
In this paper we show the convergence of a semidiscrete time stepping $\theta$-scheme on a time grid of variable length to the solution
of parabolic operator differential inclusion in the framework of evolution triple. The multifunction is assumed to be strong-weak upper-semicontinuous and to have nonempty, closed and convex values, while the quasilinear operator present in the problem is required to be pseudomonotone, coercive and satisfy the appropriate growth condition. The convergence of piecewise constant and piecewise linear interpolants constructed on the solutions of time discrete problems is shown. Under an additional assumption on the sequence of time grids and regularity of quasilinear operator strong convergence results are obtained.


\end{abstract}

\section{Introduction}

This paper deals with the numerical analysis of time-discrete scheme for the abstract operator initial value problem
for a quasilinear parabolic inclusion, where the multivalued term is a multifunction with nonempty, closed and convex values which is strong-weak upper-semicontinuous. The problem is considered in the standard framework of evolution triple of spaces $V\subset H\subset V^*$, where the embeddings are assumed to be compact and the multivalued term is defined on a reflexive Banach space $U$ which is associated with the space $V$ through a linear and compact mapping $\iota:V\to U$.
Such settings constitute a unified framework for two cases: one, where $U=H$ and $\iota$ is the embedding operator and another, where $V=H^1(\Omega)$ or its closed subspace for an open and bounded set $\Omega\subset \mathbb{R}^n$ and $\iota$ is the trace operator. In the first case the multivalued term represents the source term and in the second one, it represents the multivalued Robin type boundary conditions.

An important class of multifunctions that satisfy the assumptions presented here are the Clarke subdifferentials of locally Lipschitz functionals \cite{Clarke1993}.
Such multifunctions can be obtained by the procedure of ''filling in'' with vertical intervals the gaps in the graphs of discontinuous functions \cite{Chang1981}.

Quasilinear operators that appear in the inclusion are assumed to be coercive, pseudomonotone and satisfy an appropriate growth condition. For the inclusion
the time discrete difference $\theta$-scheme is formulated on the time mesh of variable step length. The case of the variable length time grid is important from the numerical point of view since it appears, for example, in the adaptive time stepping schemes. The convergence of a subsequence of analyzed scheme to one of possibly many solutions of the original problem is shown for $\theta\in [\frac{1}{2},1]$ using the approach based on the a priori estimates for piecewise linear and piecewise constant interpolants constructed on the time discrete solutions. The main tool used to pass to the limit in the multivalued term is the generalization of the Lions-Aubin Lemma proved in \cite{Kalita2012} (see Proposition \ref{prop:embedding} is the sequel) which is used for the sequence of piecewise constant interpolants.

Weak solutions of partial differential inclusions with nonmonotone multifunctions, in general, cannot be expected to exhibit additional regularity \cite{Kalita2012JMAA}. The results presented in the present paper show that even in nonsmooth and nonmonotone cases, a numerical approximation is still valid and one can expect the convergence of the sequence of the approximate solutions.

The paper can be considered, on one hand, as the follow up to the results of \cite{Kalita2012}, where the backward Euler scheme for the autonomous problem on equidistant time grid was studied, and the multivalued term had the form of the Clarke subgradient, and, on the other hand, as the extension of results of \cite{Emmrich2009} where the case of equation is analyzed.

The assumptions of the present paper are more general that the ones in \cite{Emmrich2009}. Quasilinear operator here is pseudomonotone, while in \cite{Emmrich2009} it is monotone and hemicontinuous. This extension allows to treat the strongly continuous perturbation as the part of the operator and therefore the convergence of the scheme
with this perturbation is proved without additional assumptions as in \cite{Emmrich2009}, where a rather restrictive condition on time grids is needed to handle
the strongly continuous perturbation (see Lemma 3 and Theorem 5 in \cite{Emmrich2009} and Theorem \ref{thm:improved} in Section \ref{sec:improved} where this condition is used to derive some improved convergence results).

The important assumption which is used here, that does not appear in \cite{Emmrich2009}, is the regularity of the sequence of time grids ($H(t)(ii)$ in what follows). This assumption is needed here only to deal with the multivalued term. Note that it was used in \cite{Carstensen1999}, where the backward Euler scheme on a variable time grid for the case of monotone multifunction (i.e. variational inequality) was analyzed. Another paper in which the thorough analysis of the monotone multifunction is delivered is \cite{Nochetto2000}, where also a posteriori error estimates are derived.

Both the quasilinear operator and multivalued term are assumed here to depend on time. They are approximated by means of the $0$-th order Cl\'{e}ment quasi-interpolation according to Remark 8.21 of \cite{Roubicek2005}. It is shown here that no additional time smoothness of the operator present in the equation, save for measurability, is needed for the scheme to converge. Note that in \cite{Emmrich2009} it is assumed that the quasilinear operator depends on time continuously. Here, under an increased time smoothness, the improved convergence result is shown: Theorem \ref{thm:improved} shows that if the quasilinear operator is H\"{o}lder continuous with respect to time, then we have pointwise strong convergence in $H$ of piecewise linear interpolants for \textit{all} $t$. We remark that the convergence for \textit{almost all} $t$ follows by Lions-Aubin Lemma but sometimes, for example while investigating the upper semicontinuous convergence of global attractors of semidiscrete schemes to the attractor of time continuous problem, one needs the pointwise convergence for all $t$ (see \cite{Kloeden2009}).

We mention that the main result of this article can be understood as the existence result, which is more general than the ones known previously for differential inclusions with Clarke subgradient (hemivariational inequalities): for example in \cite{Liu2000} only the case of multivalued source term is analyzed and in \cite{Migorski2004} the existence for the boundary case but with $p=2$ is shown (we consider here $p\in (1,\infty)$). Moreover the assumption $H(F)(iv)$ below is more general than the sign condition considered, for example, in \cite{Migorski2004} (see hypothesis $H(j)(iv)$ in \cite{Migorski2004}).

The plan of the paper is the following. In Section \ref{sec:setup} the problem setting and assumptions as well as formulations of time continuous and time discrete problems are presented. Section \ref{sec:pseudo} is devoted to some auxiliary results on pseudomonotonicity. In Section \ref{sec:main} the a priori estimates for time discrete solutions are derived and the convergence of semidiscrete scheme is proved. In Section \ref{sec:improved} some results on the strong convergence of approximate solutions are shown.

\section{Problem setting} \label{sec:setup}

Let $V\subset H\subset V^*$ be an evolution triple of spaces where all the embeddings are assumed to be continuous, dense and compact. The space $V$ is assumed to be a separable and reflexive Banach space while the space $H$ is a Hilbert space. Embedding between $V$ and $H$ will be denoted by $i:V\to H$. The norm of $V$ will be denoted without subscript while all other norms will have subscripts denoting the corresponding spaces. The duality pairing between $V^*$ and $V$ will be denoted by $\langle\cdot ,\cdot\rangle$ while for other spaces this symbol will be used with appropriate subscript. The scalar product in $H$ will be denoted by $(\cdot, \cdot)$. Furthermore, let $U$ be a reflexive Banach space and let $\iota$ be a linear, continuous and compact mapping $\iota: V\to U$. By $\iota^*:U^*\to V^*$ we denote the mapping adjoint to $\iota$ defined as $\langle\iota^*u,v\rangle=\langle u, \iota v\rangle_{U^*\times U}$. Let  $T>0$ and $1< p <\infty$. The letter $q$ will always denote the exponent conjugate to $p$, i.e. $1/p+1/q=1$. We write $\C{V}=L^p(0,T;V)$, $\C{H}=L^2(0,T;H)$, $\C{V}^*=L^q(0,T;V^*)$ and $\C{U}=L^p(0,T;U)$. The Nemytskii mappings for $\iota$ and $\iota^*$ will be denoted by the same symbols.

We will now remind several definitions of various types of pseudonomotone operators in which $X$ is always assumed to be a real and reflexive Banach space and $X^*$ denotes its dual.

\begin{definition} (see \cite{Zeidler1990}, Chapter 27) An operator $A:X\to X^*$ is called pseudomonotone, if $v_n\to v$ weakly in $X$ and $\limsup_{n\to\infty}\langle Av_n, v_n-v\rangle\leq 0$ imply that for every $y\in X$ we have $\langle Av, v-y\rangle \leq \liminf_{n\to\infty}\langle Av_n, v_n-y\rangle$.
\end{definition}

\begin{definition} (see \cite{Zeidler1990}, Chapter 27) An operator $A:X\to X^*$ is called to be of type $(S)_+$, if $v_n\to v$ weakly in $X$ and $\limsup_{n\to\infty}\langle Av_n, v_n-v\rangle\leq 0$ imply that $v_n\to v$ strongly in $X$.
\end{definition}

\noindent The next two definitions are natural generalizations of $L$-pseudo\-mono\-tonicity and $(S)_+$ property (see \cite{DMP2003Appl}, Chapter 1.3). In both of them, $W$ is a Banach space such that $W\subset X$ with a continuous embedding.

\begin{definition}
An operator $A:X\to X^*$ is called $W$-pseudomonotone (or pseudomonotone with respect to $W$), if $v_n\to v$ weakly in $X$, where $\{v_n\}_{n=1}^\infty$ is a sequence bounded in $W$ and $\limsup_{n\to\infty}\langle Av_n, v_n-v\rangle\leq 0$ imply that for every $y\in X$, we have $\langle Av, v-y\rangle \leq \liminf_{n\to\infty}\langle Av_n, v_n-y\rangle$.
\end{definition}

\begin{definition}
An operator $A:X\to X^*$ is called $W$-$(S)_+$ (or $(S)_+$ with respect to $W$), if $v_n\to v$ weakly in $X$, where $\{v_n\}_{n=1}^\infty$ is a sequence bounded in $W$ and $\limsup_{n\to\infty}\langle Av_n, v_n-v\rangle\leq 0$ imply that $v_n\to v$ strongly in $X$.
\end{definition}

\noindent The definition of pseudomonotonicity of multifunctions is not a simple generalization of single valued case.

\begin{definition} (see \cite{DMP2003Appl}, Chapter 1.3)
A multifunction $A:X\to 2^{X^*}$ is pseudomonotone, if
\begin{itemize}
\item[(i)] $A$ has values which are nonempty, weakly compact and convex,
\item[(ii)] $A$ is usc from every finite dimentional subspace of $X$ into $X^*$ furnished with weak topology,
\item[(iii)] if $v_n\to v$ weakly in $X$ and $v_n^*\in A(v_n)$ is such that $\limsup_{n\to\infty}\langle v_n^*, v_n-v\rangle\leq 0$ then for every
$y\in X$ there exists $u(y)\in A(v)$ such that $\langle u(y), v-y\rangle \leq \liminf_{n\to\infty}\langle v_n^*, v_n-y\rangle$.
\end{itemize}
\end{definition}

\noindent Note that it is useful to check pseudomonotonicity of multifunctions via following sufficient condition (see Proposition 1.3.66 in \cite{DMP2003Appl}
or Proposition 3.1 in \cite{Carl2006}).

\begin{prop}\label{prop:pseudo}A multifunction $A:X\to 2^{X^*}$ is pseudomonotone, if it satisfies the following conditions
\begin{itemize}
\item[(i)] $A$ has values which are nonempty, weakly compact and convex,
\item[(ii)] $A$ is bounded,
\item[(iii)] if $v_n\to v$ weakly in $X$ and $v_n^*\to v^*$ weakly in $X^*$ with $v_n^*\in A(v_n)$ and if $\limsup_{n\to\infty}\langle v_n^*, v_n-v\rangle\leq 0$ then
$v^*\in A(v)$ and $\langle v_n,v_n^*\rangle \to \langle v,v^*\rangle$.
\end{itemize}
\end{prop}

\noindent Next, let $q\geq 1$. We recall that $BV^{q}(0,T;X)$ is the space of functions on the time interval $[0,T]$ with values in the Banach space $X$ such that the seminorm
$\|x\|_{BV^{q}(0,T;X)}=\sup_{\pi\in\C{F}}\sum_{\sigma_i\in \pi}\|x(b_i)-x(a_i)\|_X^{q}$ is finite, where $\C{F}$ is the family of all partitions of $[0,T]$ into a finite number of disjoint subintervals $\sigma_i=(a_i,b_i)$. Moreover, for Banach spaces $X$ and $Z$ such that $X\subset Z$ and $1\leq p,\ q<\infty$, we define a Banach space
$M^{p,q}(0,T;X,Z)=L^p(0,T;X)\cap BV^{q}(0,T;Z)$. For this space the analogue of Lions-Aubin Compactness lemma holds (see Proposition 2 in \cite{Kalita2012}).

\begin{prop}\label{prop:embedding}
If $1\leq p,\ q<\infty$ and $X_1,X_2,X_3$ are Banach spaces such that embedding $X_1\subset X_2$ is compact and $X_2\subset X_3$ is continuous, then a family
of functions which is bounded in $M^{p,q}(0,T;X_1,X_3)$ is relatively compact in $L^p(0,T;X_2)$.
\end{prop}

\noindent The problem under consideration is the following:
\medskip

\noindent \textbf{Problem} $\mathbf{(\C{P})}$. Find $u\in \C{V}$ with $u'\in \C{V}^*$ such that
\begin{eqnarray}
&& u'(t)+A(t,u(t))+\iota^* F(t, \iota u(t)) \ni f(t)\quad \mbox{for a.e.}\ t\in(0,T),\\
&& u(0)=u_0.\nonumber
\end{eqnarray}
\medskip
The problem data are assumed to satisfy the following conditions.

\begin{itemize}
\item[$H(A)$]: $A:(0,T)\times V\to V^*$ is the mapping such that
\begin{itemize}
\item[$(i)$] $A(\cdot,v)$ is measurable for all $v\in V$,
\item[$(ii)$] $A(t,\cdot)$ is pseudomonotone for almost all $t\in (0,T)$,
\item[$(iii)$] $A(t,\cdot)$ is coercive for almost all $t\in (0,T)$, i.e., $\langle A(t,v),v\rangle \geq \alpha \|v\|^p - \beta \|v\|_H^2$ for all $v\in V$ and a.e $t\in (0,T)$ with $\alpha>0$ and $\beta\geq 0$,
\item[$(iv)$] $A$ satisfies the growth condition $\|A(t,v)\|_{V^*}\leq a(\|v\|_H)(1+\|v\|^{p-1})$ for all $v\in V$ and a.e $t\in (0,T)$ with $a: [0,\infty)\to [0,\infty)$ nondecreasing.
\end{itemize}
\end{itemize}

\noindent The Nemytskii operator for $A$, denoted by $\C{A}:\C{V}\to \C{V}^*$, is defined as $(\C{A}u)(t)=A(t,u(t))$ for a.e. $t\in(0,T)$ and $u\in\C{V}$.

\begin{itemize}
\item[$H(F)$]: $F:(0,T)\times U\to 2^{U^*}$ is the multifuction such that
\begin{itemize}
\item[$(i)$] $F(\cdot,u)$ is measurable for all $u\in U$,
\item[$(ii)$] the set $F(t,u)$ is nonempty, closed and convex in $U^*$ for all $u\in U$ and a.e. $t\in (0,T)$,
\item[$(iii)$] the mapping $F(t,\cdot)$ is upper semicontinuous from the strong topology of $U$ into weak topology of $U^*$ for a.e. $t\in (0,T)$,
\item[$(iv)$] at least one one of the following conditions holds
\begin{itemize}
\item[$A)$] there exists a linear and continuous mapping $p:H\to U$ such that for $v\in V$ we have $p(i(v))=\iota(v)$ and $F$ satisfies the growth condition
$\|\xi\|_{U^*}\leq c_1+d_1\|u\|_U$ for all $\xi\in F(t,u)$, all $u\in U$ and a.e. $t\in (0,T)$ with $d_1\geq 0$ and $c_1\geq 0$,
\item[$B)$] for all $u\in U$ we have $\inf_{\xi\in F(t,u)}\langle \xi,u \rangle_{U^*\times U}\geq g(t)-\lambda\|u\|_U^p$, where $0<\lambda<\frac{\alpha}{\|\iota\|_{\C{L}(V;U)}^p}$ and $g\in L^1(0,T)$, and $F$ satisfies the growth condition $\|\xi\|_{U^*}\leq c_2+d_2\|u\|^{p-1}_U$ for all $\xi\in F(t,u)$, all $u\in U$ and a.e. $t\in (0,T)$ with $d_2\geq 0$ and $c_2\geq 0$.
\end{itemize}
\end{itemize}

\item[$H_0$]: $f\in \C{V}^*$, $u_0\in H$.
\end{itemize}

\begin{rem} In both hypotheses $A)$ and $B)$ of $H(F)(iv)$ multifunction $F$ satisfies the growth condition $\|\xi\|\leq c_3+d_3\|u\|^{\max\{1,p-1\}}$ for all $\xi\in F(t,u)$, all $u\in U$ and a.e. $t\in (0,T)$ with $d_3\geq 0$ and $c_3\geq 0$. We will use the notation $r=\max\{1,p-1\}$. Also note, that in \cite{Emmrich2009}, the more general case $f=f^1+f^2$ is considered, in which $f^1\in \C{V}^*$ and $f^2\in \C{H}$. It remains an open problem whether the argument of this paper can be applied to such case.\end{rem}

\noindent We also need the following auxiliary condition concerning the space $U$ and the mapping $\iota$.

\begin{itemize}
\item[$H(U)$]: the mapping $\iota$ is linear, continuous and compact, and the Nemytskii mapping $\hat{\iota}:M^{p,q}(0,T;V,V^*)\to \C{U}$ defined by $(\hat{\iota}u)(t)=\iota(u(t))$ is also compact.
\end{itemize}

\noindent Problem $\mathbf{(\C{P})}$ will be approximated by means of a semidiscrete $\theta$-scheme on a time grid of a variable time step length.
To this end, let us define the sequence of grids indexed by $n\in \mathbb{N}$:
$$
\C{T}_n=\{0=t^0_n<t^1_n <\ldots < t^{N_n}_n=T\}
$$
Moreover we define $\tau^k_n=t^k_n-t^{k-1}_n$ for $k\in \{1,\ldots,N_n\}$. We introduce the notation $\tau^{max}_n = \max_{k=1,\ldots,N_n}\tau^k_n$ and $\tau^{min}_n = \min_{k=1,\ldots,N_n}\tau^k_n$. We need the following assumption on the time grid.

\begin{itemize}
\item[$H(t)$]: the sequence of time grids satisfies \begin{itemize}
\item[$(i)$] $\lim_{n\to\infty}\tau^{max}_n = 0$,
\item[$(ii)$] there exists the constant $K>0$ such that $\tau^{max}_n\leq K \tau^{min}_n$ for all $n\in \mathbb{N}$.
\end{itemize}
\end{itemize}

\noindent We remark that the sequence of time grids that satisfies $H(t)(ii)$ is called regular (see \cite{Carstensen1999}). For brevity of notation for a reflexive Banach space $X$ and $s\geq 1$ we introduce the operator $$\pi^{s,X}_n: L^s(0,T;X)\to L^s(0,T;X),$$ defined as $$(\pi^{s,X}_n(v))(t)= \frac{1}{\tau^k_n} \int_{t_n^{k-1}}^{t_n^k} v(t)\ dt \ \ \mbox{for}\ \  t\in (t^{k-1}_n,t^k_n].$$ Lemma 3.3 in \cite{Carstensen1999} states that if $H(t)$ holds, then $\pi^{s,X}_n v \to v$ strongly in $L^s(0,T;X)$ as $n\to\infty$ (in fact only $H(t)(i)$ is needed).

\noindent The initial condition $u_0$ will be approximated by means of a sequence of elements of $V$, namely, we need the sequence $\{u_{0n}\}_{n=1}^{\infty}$ such that $u_{0n}\in V$ and $u_{0n}\to u_0$ strongly in $H$ as $n\to \infty$.

\noindent The semidiscrete scheme consists in the recursive solving of of the approximate problem. Obtained solutions correspond
to the values in the points of the time mesh. In order to formulate the approximate problem, we need the to define the auxiliary quantities
$$
A^k_n:V\to V^*,\  f^k_n\in V^*,\  F^k_n:U\to 2^{U^*}\ \mbox{for}\ k\in 1,\ldots,N_n
$$
by
\begin{eqnarray}
&&A^k_n(u)=\frac{1}{\tau^k_n}\int_{t^{k-1}_n}^{t^k_n} A(t,u)\ dt,\  f^k_n=\frac{1}{\tau^k_n}\int_{t^{k-1}_n}^{t^k_n}f(t)\ dt,\\
&&F^k_n(u)=\left\{\frac{1}{\tau^k_n}\int_{t^{k-1}_n}^{t^k_n}\xi(t)\ dt\right\},
\end{eqnarray}
where $\xi(t)\in F(t,u)\ \mbox{for a.e.}\ t\in(t^{k-1}_n,t^k_n)$.

\noindent Now we fix the scheme parameter $\theta\in (0,1]$. We are ready to formulate the time discretized problem. Let $n\in \mathbb{N}$.
\medskip

\noindent \textbf{Problem} $\mathbf{(\C{P}^n)}$. Find $u^k_n\in V$ for $k=1,\ldots, N_n$ such that
\begin{equation}
\frac{u^k_n-u^{k-1}_n}{\tau^k_n}+A^k_n(u^{k-1+\theta}_n)+\iota^* F^k_n(\iota u^{k-1+\theta}_n)\ni f^k_n,\label{eq:theta_scheme}
\end{equation}
where $u^{k-1+\theta}_n=\theta u^k_n+(1-\theta)u^{k-1}_n$ with $u^0_n=u_{0n}$.
\medskip

\noindent The choice of parameter $\theta=1$ in Problem $\mathbf{(\C{P}^n)}$ corresponds to implicit Euler method also known as the Rothe method, while the
case $\theta=\frac{1}{2}$ corresponds to the Crank-Nicholson scheme. Note that the case $\theta=0$ which corresponds to the explicit Euler scheme is excluded since in such case there is no guarantee
that the solution $u^k_n$ of (\ref{eq:theta_scheme}), which can then be explicitly calculated, belongs to the space $V$.

\section{Auxiliary results on pseudomonotonicity}\label{sec:pseudo}

This section is devoted to results on pseudomonotonicity and $(S)_+$ property of auxiliary operators that appear in problems under consideration.

\begin{lemma} \label{lemma:pseudomon}Under assumption $H(A)$, the operators $A^k_n$ are pseudomonotone (in the sense of multifunctions) for all $n\in\mathbb{N}$ and $k\in\{1,\ldots,N_n\}$.
\end{lemma}
\begin{proof} First we fix the indices $n\in \mathbb{N}$ and $k\in\{1,\ldots,N_n\}$. Observe that for $v\in V$, we have
$$
\|A^k_nv\|_{V^*}\leq \frac{1}{\tau^k_n}\int_{t^{k-1}_n}^{t^k_n}\|A(t,v)\|_{V^*}\, dt\leq
$$
$$\leq \frac{1}{\tau^k_n}\int_{t^{k-1}_n}^{t^k_n}a(\|v\|_H)(1+\|v\|^{p-1})\, dt
\leq a(\|v\|_H)(1+\|v\|^{p-1}).
$$
This estimate shows that the operator $A^k_n$ is bounded. Now, let us take $v_m\to v$ weakly in $V$, as $m\to\infty$ with
$\limsup_{m\to\infty} \langle A^k_n v_m, v_m-v\rangle \leq 0$. We proceed by a standard argument of \cite{Berkovits1996}. Obviously $v_m\to v$ strongly in $H$ and $\|v_m\|_H\leq R$ for all $m\in\mathbb{N}$ with $R>0$. Define $\xi_m(t)=\langle A(t,v_m),v_m-v \rangle$ and $C=\{t\in (t^{k-1}_n,t^k_n): \liminf_{m\to\infty} \xi_m(t)<0 \}$. The latter is the Lebesgue measurable subset of $(t^{k-1}_n,t^k_n)$. Assume that $m(C)>0$ and pick $t\in C$. For a subsequence which is still denoted by $m$, we have $\lim_{m\to\infty}\langle A(t,v_m), v_m-v\rangle < 0$. From $H(A)(iii)$ we have $0\leq\liminf_{m\to\infty}\langle A(t,v_m),v_m-v\rangle$ which is a contradiction. This means that we have $\liminf_{m\to\infty} \xi_m(t)\geq 0$ for a.e. $t\in (t^{k-1}_n,t^k_n)$. Now, we have the estimate
$$
\xi_m(t)\geq \alpha \|v_m\|^p - \beta \|v_m\|_H^2 - a(R) \|v\|-a(R)\|v\|\|v_m\|^{p-1}.
$$
Next, the Young inequality with $\epsilon$ gives $$a(R)\|v\|\|v_m\|^{p-1}\leq \epsilon \|v_m\|^p+C(\epsilon)(a(R))^p\|v\|^p,$$ which, by taking $\epsilon=\frac{\alpha}{2}$, leads to
$$
\xi_m(t)\geq \frac{\alpha}{2} \|v_m\|^p - \beta \|v_m\|_H^2 - a(R) \|v\|-C\left(\frac{\alpha}{2}\right)(a(R))^p\|v\|^p .
$$
We are now in position to use the Fatou lemma to obtain
$$
\beta\|v\|_H^2\leq \frac{1}{\tau^k_n}\int_{t^{k-1}_n}^{t^k_n} \liminf_{n\to\infty} \xi_m(t) \, dt+ \beta\|v\|_H^2\leq
$$
$$
\leq \frac{1}{\tau^k_n} \liminf_{n\to\infty}\int_{t^{k-1}_n}^{t^k_n} \xi_m(t) + \beta\|v_m\|_H^2\, dt\leq \beta \|v\|_H^2+\limsup_{n\to\infty}\frac{1}{\tau^k_n}\int_{t^{k-1}_n}^{t^k_n} \xi_m(t) \, dt=
$$
$$
=\beta \|v\|_H^2+\limsup_{n\to\infty}\langle A^k_n v_m,v_m-v\rangle\leq \beta \|v\|_H^2.
$$

\noindent Thus we have $\int_{t^{k-1}_n}^{t^k_n}\xi_m(t)\, dt \to 0$ as $m\to \infty$. Now note that $|\xi_m(t)|=\xi_m(t)+2\xi_m^-(t)$ and $\xi_m^-(t)\to 0$ for a.e. $t\in (t^{k-1}_n,t^k_n)$. Since $\xi_m(t)+\beta\|v_m\|_H^2\geq h(t)$ with $h\in L^1(t^{k-1}_n,t^k_n)$, then also $\xi_m^-(t)-\beta\|v_m\|_H^2\leq h^-(t)$ and we can invoke Fatou lemma again to get $\limsup_{m\to\infty}\int_{t^{k-1}_n}^{t^k_n}\xi_m^-(t)\, dt\leq 0$ and moreover $\int_{t^{k-1}_n}^{t^k_n}\xi_m^-(t)\, dt\to 0$, as $m\to \infty$. We deduce that $\xi_m\to 0$ in $L^1(t^{k-1}_n,t^k_n)$, and, for a subsequence, still denoted by the same index, we have $\xi_m(t)\to 0$ for a.e. $t\in (t^{k-1}_n,t^k_n)$. By the pseudomonotonicity of $A(t,\cdot)$, it follows that for a.e. $t\in (t^{k-1}_n,t^k_n)$, we have $A(t,v_m)\to A(t,v)$ weakly in $V^*$ and $\langle A(t,v_m),v_m\rangle \to \langle A(t,v),v\rangle$. Let $u\in V$. We have
$$
\langle A^k_n v, v-u\rangle = \frac{1}{\tau^k_n}\int_{t^{k-1}_n}^{t^k_n} \langle A(t,v), v-u\rangle \, dt =
$$
$$
= \frac{1}{\tau^k_n}\int_{t^{k-1}_n}^{t^k_n} \lim_{m\to\infty}\langle A(t,v_m), v_m-u\rangle \, dt =
$$
$$
= -\beta \|v\|_H^2 + \frac{1}{\tau^k_n}\int_{t^{k-1}_n}^{t^k_n} \lim_{m\to\infty}(\langle A(t,v_m), v_m-u\rangle +\beta\|v_m\|_H^2)\, dt.
$$
Invoking Fatou lemma one last time, we get
$$
\langle A^k_n v, v-u\rangle \leq \liminf_{m\to\infty}\frac{1}{\tau^k_n}\int_{t^{k-1}_n}^{t^k_n} \langle A(t,v_m), v_m-v+v-u\rangle \, dt=
$$
$$
= \liminf_{m\to\infty}\frac{1}{\tau^k_n}\int_{t^{k-1}_n}^{t^k_n} \langle A(t,v_m), v-u\rangle \, dt=\liminf_{m\to\infty}\langle A^k_n v_m,v-u\rangle.
$$
The assertion of the lemma is proved.
\end{proof}

\begin{lemma}\label{lemma:pseudo:nemytskii} Under assumption $H(A)$, the Nemytskii operator $\C{A}:\C{V}\to\C{V}^*$ corresponding to $A$ is $M^{p,q}(0,T;V,V^*)$-pseudomonotone.
\end{lemma}
\begin{proof}
The proof is omitted since it exactly follows the lines of the proof of Lemma 1 in \cite{Kalita2012}, where the autonomous case was considered. The main lines of the argument follow that of Lemma \ref{lemma:pseudomon}. Compare also
proof of Theorem 2(b) in \cite{Berkovits1996} and Lemma 8.8 in \cite{Roubicek2005}.
\end{proof}

\begin{lemma} If, in addition to hypothesis $H(A)$, the operators $A(t,\cdot)$ are of type $(S)_+$ for a.e. $t\in (0,T)$, then the Nemytskii operator $\C{A}$ is of type $M^{p,q}(0,T;V,V^*)$-$(S)_+$. \label{lemma:splus}
\end{lemma}
\begin{proof}
The proof is omitted since it follows the lines of the proof of Lemma 2 in \cite{Kalita2012JMAA}, where the autonomous case was considered. Compare also
the proof of Theorem 2(c) in \cite{Berkovits1996}.
\end{proof}

\begin{lemma}
Under assumption $H(F)$, the multifunctions $\iota^* F^k_n(\iota \cdot): V\to 2^{V^*}$ are pseudomonotone for all $n\in\mathbb{N}$ and $k\in \{1,\ldots,N_n\}$.
\end{lemma}
\begin{proof}
\noindent  We fix the indices $n\in \mathbb{N}$ and $k\in \{1,\ldots,N_n\}$. First we show boundedness. For $v\in V$ and $\eta\in \iota^* F^k_n(\iota v)$, we have $\eta=\iota^* \xi$ with $\xi\in F^k_n(\iota v)$, i.e. $\xi=\frac{1}{\tau^k_n}\int_{t^{k-1}_n}^{t^k_n}\zeta(t)\ dt$ with $\zeta (t)\in F(t, \iota v)$, and
\begin{equation}\label{eqn:estgk}
\|\eta\|_{V^*}\leq \|\iota\|_{\C{L}(V,U)}\|\xi\|_{U^{*}}\leq \|\iota\|_{\C{L}(V,U)}(c_3+d_3\|\iota\|_{\C{L}(V,U)}^r\|v\|^r).
\end{equation}
Obviously, the set $\iota^* F^k_n(\iota v)$ is nonempty. Its convexity follows from the fact that $F(t,v)$ is convex a.e. $t$. To show closedness, assume that $\eta_m\in \iota^* F^k_n(\iota v)$ and $\eta_m\to \eta$ strongly in $V^*$. We have $\eta_m=\iota^* \xi_m$ with $\xi_m\in F^k_n(\iota v)$, i.e. $\xi_m=\frac{1}{\tau^k_n}\int_{t^{k-1}_n}^{t^k_n}\zeta_m(t)\ dt$ with $\zeta_m(t)\in F(t, \iota v)$. From $H(F)(iv)$, we have $\|\zeta_m(t)\|_{U^*}\leq c_3+d_3\|\iota\|_{\C{L}(V,U)}^r\|v\|^r$, from which it follows that, for
a subseqence, $\zeta_m\to \zeta$ weakly in $L^q(t^{k-1}_n,t^k_n;U^*)$, and weakly in $L^1(t^{k-1}_n,t^k_n;U^*)$. We can invoke the convergence theorem of
Aubin and Cellina (see \cite{Aubin1984} and Proposition 2 in \cite{Migorski2009}) and we obtain that $\zeta(t)\in F(t,\iota v)$ for a.e. $t$.
Moreover, for $w\in U$ we have
$$
\langle\xi_m,w\rangle_{U^*\times U}=\frac{1}{\tau^k_n}\int_{t^{k-1}_n}^{t^k_n}\langle \zeta_m(t), w\rangle_{U^*\times U}\ dt\to \frac{1}{\tau^k_n}\int_{t^{k-1}_n}^{t^k_n}\langle \zeta(t), w\rangle_{U^*\times U}\ dt=
$$
$$
=\left\langle\frac{1}{\tau^k_n}\int_{t^{k-1}_n}^{t^k_n} \zeta(t)\ dt, w\right\rangle_{U^*\times U},
$$
 which means that $\xi_m\to\xi$ weakly in $U^*$, where $\xi\in F^k_n(\iota v)$. Hence $\eta\in \iota^* F^k_n(\iota v)$.
It remains to show the condition $(iii)$ of Proposition \ref{prop:pseudo}. Assume that $v_m\to v$ weakly in $V$ and $\eta_m\in \iota^* F^k_n(\iota v_m)$ with $\eta_m\to \eta$ weakly in $V^*$. By compactness of $\iota$, we have $\iota v_m\to \iota v$ strongly in $U$. We proceed similarily as in the proof of closedness. We have $\eta_m=\iota^* \xi_m$ with  $\xi_m=\frac{1}{\tau^k_n}\int_{t^{k-1}_n}^{t^k_n}\zeta_m(t)\ dt$ and $\zeta_m(t)\in F(t, \iota v_m)$. By the growth condition $H(F)(iv)$ we deduce that $\|\zeta_m(t)\|_{U^*}$ is bounded by a constant independent on $m$ and $t$ and we can extract from $\zeta_m$ a subsequence if necessary that converges weakly in $L^q(t^{k-1}_n,t^k_n;U^*)$ to some $\zeta$. We can invoke convergence theorem of Aubin and Cellina again to get $\zeta(t)\in F(t, \iota v)$ for a.e. $t$. If we put $\xi=\frac{1}{\tau^k_n}\int_{t^{k-1}_n}^{t^k_n} \zeta(t)\ dt$, then $\xi_m\to \xi$ weakly in $U^*$ and $\eta_m\to \iota^*\xi$ weakly in $V^*$. Hence $\iota^*\xi=\eta$ and the convergence holds for the whole sequence. We have shown that $\eta\in \iota^* F^k_n(\iota v)$. Finally he have $\langle \eta_m,v_m\rangle = \langle \xi_m, \iota v_m\rangle_{U^*\times U} \to \langle \xi, \iota v\rangle_{U^*\times U}=\langle \eta,v\rangle$ and the proof is complete.
\end{proof}

\section{Convergence of semidiscrete scheme}\label{sec:main}

In this section, first we formulate the result which guarantees the existence of solutions to semidiscrete problems and then we proceed
with a priori estimates and passing to the limit which is shown to solve the time continuous problem.

\begin{lemma}
Let $n\in\mathbb{N}$ and $k\in\{1,\ldots,N_n\}$ be given. Under assumptions $H(A)$, $H(F)$, $H_0$ and $H(U)$, there exists $\tau_0>0$ such that if $0< \tau^k_n < \tau_0$ then there exists $u^k_n\in V$ solution to Problem $\mathbf{(\C{P}^n)}$.
\end{lemma}
\begin{proof}
We rewrite equivalently (\ref{eq:theta_scheme}) as follows
$$
\frac{1}{\theta \tau^k_n} u^{k-1+\theta}_n + A^k_n u^{k-1+\theta}_n + \iota^* F^k_n(\iota u^{k-1+\theta}_n) \ni \frac{1}{\theta \tau^k_n}u^{k-1}_n+f^k_n.
$$
We show that, given $u^{k-1}_n\in V$, there exists $u^{k-1+\theta}_n$ that satisfies the above inclusion. We prove that the range of multifunction $V\ni v\to Lv = \frac{i^*iv}{\theta \tau^k_n}  + A^k_n v + \iota^* F^k_n(\iota v)$ is the whole space
$V^*$. This will be done by a surjectivity theorem of Br\'{e}zis (see, for instance, Theorem 1.3.70 in \cite{DMP2003Appl}). We need to show that $L$ is coercive (in the sense that $\lim_{\|v\|\to\infty}\inf_{v^*\in Lv}\frac{\langle v^*,v\rangle}{\|v\|}=\infty$) and pseudomonotone. Since the operator $\frac{i^*i}{\theta \tau^k_n}$ satisfies conditions $(i)$-$(iii)$ of Proposition \ref{prop:pseudo} trivially, and we already know that $A^k_n$ and
$v \to \iota^* F^k_n(\iota v)$ are pseudomonotone, the pseudomonotonicity of $L$ follows from  the fact that sum of pseudomonotone multifunctions is pseudomonotone, cf. \cite{DMP2003Appl} Proposition 1.3.68. In order to show the coercivity of $L$ we need to assume that $v^*\in Lv$ and estimate $\langle v^*,v\rangle$ from below. We have, with some $\eta\in F^k_n(\iota v)$
\begin{eqnarray}
&&\langle v^*,v\rangle \geq \frac{\|v\|_H^2}{\theta \tau^k_n}+\frac{1}{\tau^k_n}\int_{t^{k-1}_n}^{t^k_n}(\alpha\|v\|^p-\beta\|v\|_H^2)\ dt+\langle \eta,\iota v\rangle_{U^*\times U}\geq\nonumber\\
&&\geq\|v\|_H^2\left(\frac{1}{\theta \tau^k_n}-\beta\right)+\alpha\|v\|^p+\frac{1}{\tau^k_n}\int_{t^{k-1}_n}^{t^k_n}\langle \zeta(t),\iota v\rangle_{U^*\times U}\ dt,\nonumber
\end{eqnarray}
where $\zeta(t)\in F(t,\iota v)$ for a.e. $t\in (t^{k-1}_n,t^k_n)$. We proceed with the proof separately for the cases $A)$ and $B)$ of $H(F)(iv)$. In the case $A)$, we have
\begin{eqnarray}
&&\langle \zeta(t),\iota v\rangle_{U^*\times U} \geq - (c_1+d_1\|\iota v\|_U)\|p\|_{\C{L}(H;U)}\|v\|_H \geq\nonumber\\
&&\geq  - c_1\|p\|_{\C{L}(H;U)}\|v\|_H - d_1 \|p\|_{\C{L}(H;U)}^2\|v\|_H^2\geq\nonumber\\
&&\geq - \frac{4c_1^2}{d_1\varepsilon}- (d_1+\varepsilon) \|p\|_{\C{L}(H;U)}^2\|v\|_H^2,\nonumber
\end{eqnarray}
where $\varepsilon>0$ is arbitrary. We obtain
$$
\langle v^*,v\rangle \geq \|v\|_H^2\left(\frac{1}{\theta \tau^k_n}-\beta-d_1\|p\|_{\C{L}(H;U)}^2-\varepsilon\|p\|_{\C{L}(H;U)}^2\right)+\alpha\|v\|^p - \frac{4c_1^2}{d_1\varepsilon}.
$$
Obviously, if $\tau^k_n<\frac{1}{\theta(\beta + d_1\|p\|_{\C{L}(H;U)}^2)}$, then it is possible to choose $\varepsilon$ such that the term with $\|v\|_H^2$ is nonnegative and we obtain the coercivity. In the case $B)$, by integrating the inequality in $H(F)(iv) B)$, we get
\begin{equation}\label{eqn:hfivb}
\frac{1}{\tau^k_n}\int_{t^{k-1}_n}^{t^k_n}\langle \zeta(t),\iota v\rangle_{U^*\times U}\ dt \geq - \frac{\|g\|_{L^1(t^{k-1}_n,t^k_n)}}{\tau^k_n}-\lambda\|\iota\|_{\C{L}(V;U)}^p\|v\|^p.
\end{equation}
Next, we obtain
$$
\langle v^*,v\rangle \geq\|v\|_H^2\left(\frac{1}{\theta \tau^k_n}-\beta\right)+\left(\alpha-\lambda\|\iota\|_{\C{L}(V;U)}^p\right)\|v\|^p-\frac{\|g\|_{L^1(t^{k-1}_n,t^k_n)}}{\tau^k_n}.
$$
Now, if $\tau^k_n\leq \frac{1}{\theta \beta}$, then the term with $\|v\|_H^2$ is nonegative and we obtain coercivity. This concludes the proof of the Lemma.
\end{proof}

\begin{remark}Note that if $\theta=1$, then the assumption that $u^{k-1}_n\in V$ is not needed. The solution $u^k_n$ exists if $u^{k-1}_n\in H$.
\end{remark}

\noindent Next result establishes estimates which are satisfied by the solutions of the semidiscrete problem.

\begin{lemma}\label{lemma:estimates}
Let $n\in\mathbb{N}$ be fixed. Under assumptions $H(A), H(F), H(U)$ and $H_0$ there exists $\tau_0>0$ such that if $\tau^{max}_n< \tau_0$, then following estimate holds with a constant $M$
which depends only on the problem data
$$
\max_{k=1,\ldots,N}\|u^k_n\|_H^2+\sum_{k=1}^{N_n}\tau^k_n\|u^{k-1+\theta}_n\|^p + (2\theta-1)\sum_{k=1}^{N_n}\|u^k_n-u^{k-1}_n\|_H^2\leq M.
$$
\end{lemma}
\begin{proof}
The estimates are derived by testing (\ref{eq:theta_scheme}) with $u^{k-1+\theta}_n$ and using the algebraic relation valid for all $a,b\in\mathbb{R}$
$$
(a-b)(\theta a+ (1-\theta)b)=\frac{1}{2}(a^2-b^2+(2\theta-1)(a-b)^2).
$$
For $k=1,\ldots,N_n$ we have
\begin{eqnarray}
&&\frac{1}{2\tau^k_n}(\|u^k_n\|_H^2-\|u^{k-1}_n\|_H^2+(2\theta-1)\|u^k_n-u^{k-1}_n\|_H^2)+\langle A^k_n u^{k-1+\theta}_n, u^{k-1+\theta}_n\rangle +\nonumber\\
&&+ \langle\xi^k_n, \iota  u^{k-1+\theta}_n \rangle_{U^*\times U} = \frac{1}{\tau^k_n}\int_{t^{k-1}_n}^{t^k_n}\langle f(t),u^{k-1+\theta}_n\rangle\ dt,\nonumber
\end{eqnarray}
where $\xi^k_n\in F^k_n(\iota u^{k-1+\theta}_n)$. By $H(A)(iii)$  we arrive at the estimate
\begin{eqnarray}
&&\frac{1}{2\tau^k_n}(\|u^k_n\|_H^2-\|u^{k-1}_n\|_H^2+(2\theta-1)\|u^k_n-u^{k-1}_n\|_H^2)+\alpha\|u^{k-1+\theta}_n\|^p-\\
&&-\beta\|u^{k-1+\theta}_n\|^2_H + \langle\xi^k_n, \iota  u^{k-1+\theta}_n \rangle_{U^*\times U} \leq \frac{1}{\tau^k_n}\int_{t^{k-1}_n}^{t^k_n} \|f(t)\|_{V^*}\|u^{k-1+\theta}_n\|\ dt.\nonumber
\end{eqnarray}
Using the Cauchy and Young inequalities, we get
\begin{eqnarray}
&&\|u^k_n\|_H^2-\|u^{k-1}_n\|_H^2+(2\theta-1)\|u^k_n-u^{k-1}_n\|_H^2+\nonumber\\
&&+2\tau^k_n(\alpha-\varepsilon_1)\|u^{k-1+\theta}_n\|^p-2\tau^k_n\beta\|u^{k-1+\theta}_n\|^2_H +\nonumber\\
&&+ 2\tau^k_n\langle\xi^k_n, \iota  u^{k-1+\theta}_n \rangle_{U^*\times U} \leq C_1(\varepsilon_1)\|f\|^q_{L^q(t^{k-1}_n,t^k_n;V^*)},\nonumber
\end{eqnarray}
with arbitrary $\varepsilon_1>0$ and a positive constant $C_1(\varepsilon)$, independent of $n, k$ and $\tau^k_n$.
We proceed separately for the cases $A)$ and $B)$ of $H(F)(iv)$. In the first case, with arbitrary $\varepsilon_2>0$ and the constant $C_2(\varepsilon_2)$ we have
$$
|\langle\xi^k_n, \iota  u^{k-1+\theta}_n \rangle_{U^*\times U}|\leq (d_1\|p\|_{\C{L}(U;H)}^2+\varepsilon_2)\|u^{k-1+\theta}_n\|_H^2+C_2(\varepsilon_2).
$$
Next, we can take $\varepsilon_1=\frac{\alpha}{2}$ and denote $C_1(\frac{\alpha}{2})=C_3$ to get
\begin{eqnarray}
&&\|u^k_n\|_H^2-\|u^{k-1}_n\|_H^2+(2\theta-1)\|u^k_n-u^{k-1}_n\|_H^2+\nonumber\\
&&+\tau^k_n\alpha\|u^{k-1+\theta}_n\|^p\leq 2\tau^k_n(\beta+d_1\|p\|_{\C{L}(U;H)}^2+\varepsilon_2)\|u^{k-1+\theta}_n\|^2_H + \nonumber\\
&& + C_3\|f\|^q_{L^q(t^{k-1}_n,t^k_n;V^*)}+2\tau^k_n C_2(\varepsilon_2).\label{eqn:first_case}
\end{eqnarray}
In the second case by (\ref{eqn:hfivb}), we have
\begin{eqnarray}
&&\|u^k_n\|_H^2-\|u^{k-1}_n\|_H^2+(2\theta-1)\|u^k_n-u^{k-1}_n\|_H^2+\nonumber\\
&&+2\tau^k_n(\alpha-\varepsilon_1-\lambda\|\iota\|^p_{\C{L}(V;U)})\|u^{k-1+\theta}_n\|^p\leq 2\tau^k_n\beta\|u^{k-1+\theta}_n\|^2_H +\nonumber\\
&& + C_1(\varepsilon_1)\|f\|^q_{L^q(t^{k-1}_n,t^k_n;V^*)}+ 2\|g\|_{L^1(t^{k-1}_n,t^k_n)}.\nonumber
\end{eqnarray}
Let $\varepsilon_1 = \frac{1}{2}(\alpha-\lambda\|\iota\|^p_{\C{L}(V;U)})$ and $C_4=C_1(\frac{1}{2}(\alpha-\lambda\|\iota\|^p_{\C{L}(V;U)}))$. We have
\begin{eqnarray}
&&\|u^k_n\|_H^2-\|u^{k-1}_n\|_H^2+(2\theta-1)\|u^k_n-u^{k-1}_n\|_H^2+\nonumber\\
&&+\tau^k_n(\alpha-\lambda\|\iota\|^p_{\C{L}(V;U)})\|u^{k-1+\theta}_n\|^p\leq 2\tau^k_n\beta\|u^{k-1+\theta}_n\|^2_H +\nonumber\\
&& + C_4\|f\|^q_{L^q(t^{k-1}_n,t^k_n;V^*)}+ 2\|g\|_{L^1(t^{k-1}_n,t^k_n)}.\label{eqn:second_case}
\end{eqnarray}
Now, by convexity of the mapping $H\ni u\to \|u\|_H^2$, we have
$$
\|u^{k-1+\theta}_n\|_H^2\leq \theta\|u^k_n\|_H^2+(1-\theta)\|u^{k-1}_n\|_H^2,
$$
and we are in position to use the discrete Gronwall-type lemma (see Lemma 1 in \cite{Emmrich2009}).
In the case $A)$, which leads to (\ref{eqn:first_case}), it suffices to take $\tau_0$ such that $2 \tau_0\theta(\beta+d_1\|p\|_{\C{L}(U;H)}^2)<1$ and choose
$\varepsilon_2$ such that $2 \tau_0 \theta(\beta+d_1\|p\|_{\C{L}(U;H)}^2+\varepsilon_2) <1$. We get
\begin{eqnarray}
&&\max_{k=1,\ldots,N_n}\|u^k_n\|_H^2+\sum_{k=1}^{N_n}\tau^k_n\|u^{k-1+\theta}_n\|^p + (2\theta-1)\sum_{k=1}^{N_n}\|u^k_n-u^{k-1}_n\|_H^2\leq\nonumber\\
&&\leq C_5 (1+\|u_{0n}\|_H^2+\|f\|^q_{\C{V}^*}),\nonumber
\end{eqnarray}
with $C_5>0$. In the case $B)$, which leads to (\ref{eqn:second_case}), it suffices to take $\tau_0$ such that $2 \tau_0\theta\beta<1$. We get
\begin{eqnarray}
&&\max_{k=1,\ldots,N_n}\|u^k_n\|_H^2+\sum_{k=1}^{N_n}\tau^k_n\|u^{k-1+\theta}_n\|^p + (2\theta-1)\sum_{k=1}^{N_n}\|u^k_n-u^{k-1}_n\|_H^2\leq\nonumber\\
&&\leq C_6 (\|g\|_{L^1(0,T)}+\|u_{0n}\|_H^2+\|f\|^q_{\C{V}^*} ),\nonumber
\end{eqnarray}
with $C_6>0$.
\end{proof}

\begin{lemma}\label{lemma:estimates2}
Let $n\in\mathbb{N}$ be given. Under assumptions $H(A), H(F), H(U), H_0$ and $\theta\in [\frac{1}{2},1]$, there exists $\tau_0>0$ such that if $\tau^{max}_n< \tau_0$, then the following estimates holds with the constant $M$
dependent only on the problem data
\begin{equation}\label{eqn:estimate_ak}
\sum_{k=1}^{N_n}\tau^k_n\|A^k_n u^{k-1+\theta}_n\|_{V^*}^q \leq M,
\end{equation}
\begin{equation}\label{eqn:estimate_fk}
\sum_{k=1}^{N_n}\tau^k_n\|\iota^*\xi^k_n\|_{V^*}^q \leq M,
\end{equation}
\begin{equation}\label{eqn:estimate_der}
\sum_{k=1}^{N_n}\tau^k_n\left\|\frac{u^k_n-u^{k-1}_n}{\tau^k_n}\right\|_{V^*}^q \leq M,
\end{equation}
where $\xi^k_n\in F^k_n(\iota u^{k-1+\theta})$ are such that (\ref{eq:theta_scheme}) holds.
\end{lemma}
\begin{proof}
From the Jensen inequality, we have
\begin{eqnarray}
&&\sum_{k=1}^{N_n}\tau^k_n\|A^k_n u^{k-1+\theta}_n\|_{V^*}^q  \leq \sum_{k=1}^{N_n}\tau^k_n a(\|u^{k-1+\theta}_n\|_H)^q(1+\|u^{k-1+\theta}_n\|_{V^*}^{p-1})^q\leq\nonumber \\
&&\leq 2^{q-1} \sum_{k=1}^{N_n}\tau^k_n a(\|u^{k-1+\theta}_n\|_H)^q(1+\|u^{k-1+\theta}_n\|_{V^*}^p).\label{eqn:estimate_a}
\end{eqnarray}
Observe that if $\theta\in [\frac{1}{2},1]$, then $\max_{k=1,\ldots,N_n}\|u^k_n\|_H$ is bounded from Lemma \ref{lemma:estimates}. Moreover,
$\|u_{0n}\|_H$ is bounded since this sequence approximates the initial condition in Problem $\mathbf{(\C{P})}$. Thus, we can say that
$\max_{k=0,\ldots,N_n}\|u^k_n\|_H \leq R$ with $R>0$. Hence
$$
\sum_{k=1}^{N_n}\tau^k_n\|A^k_n u^{k-1+\theta}_n\|_{V^*}^q  \leq
 2^{q-1}a(R)^q \left(T  + \sum_{k=1}^{N_n}\tau^k_n \|u^{k-1+\theta}_n\|_{V^*}^p\right).
$$
Since the last sum is bounded from the previous lemma, we obtain (\ref{eqn:estimate_ak}).

In order to establish the estimate (\ref{eqn:estimate_fk}) we observe that
\begin{equation}
\sum_{k=1}^{N_n}\tau^k_n\|\iota^*\xi^k_n\|_{V^*}^q \leq \sum_{k=1}^{N_n}\tau^k_n\|\xi^k_n\|_{U^*}^q\|\iota\|_{\C{L}(V;U)}^q.
\end{equation}
For all $k\in \{1,\ldots,N_n\}$ and a.e. $t\in (t^{k-1}_n, t^k_n)$ there exists $\xi(t)\in F(t,\iota u^{k-1+\theta}_n)$ such that $\xi^k_n=\frac{1}{\tau^k_n}\int_{t^{k-1}_n}^{t^k_n}\xi(t)\, dt$. Using the Jensen inequality, we get
\begin{equation}\label{eqn:multivalued_estimate}
\sum_{k=1}^{N_n}\tau^k_n\|\iota^*\xi^k_n\|_{V^*}^q \leq  \|\iota\|_{\C{L}(V;U)}^q \sum_{k=1}^{N_n}\int_{t^{k-1}_n}^{t^k_n}\|\xi(t)\|_{U^*}^q\, dt.
\end{equation}

In the case $A)$ of $H(F)(iv)$ we have
\begin{equation}\label{eqn:casea}
\|\xi(t)\|_{U^*}^q\leq 2^{q-1}(c_1^q+\|\iota u^{k-1+\theta}_n\|^q)\leq 2^{q-1}(c_1^q+\|p\|_{\C{L}(U;H)}^q\|u^{k-1+\theta}_n\|_H^q),
\end{equation}
for a.e. $t\in (t^{k-1}_n, t^k_n)$. The right-hand side of (\ref{eqn:casea}) is bounded by a constant by Lemma \ref{lemma:estimates} and, substituting this expression into (\ref{eqn:multivalued_estimate}),
we obtain the desired estimate.

If the hypothesis $B)$ of $H(F)(iv)$ holds, then we obtain
\begin{equation}\label{eqn:caseb}
\|\xi(t)\|_{U^*}^q\leq 2^{q-1}(c_1^q+\|\iota u^{k-1+\theta}_n\|^q)\leq 2^{q-1}(c_2^q+\|\iota\|_{\C{L}(V;U)}^q\|u^{k-1+\theta}_n\|^p),
\end{equation}
for a.e. $t\in (t^{k-1}_n, t^k_n)$. Substituting (\ref{eqn:caseb}) into (\ref{eqn:multivalued_estimate}), we get
\begin{equation}
\sum_{k=1}^{N_n}\tau^k_n\|\iota^*\xi^k_n\|_{V^*}^q \leq  \|\iota\|_{\C{L}(V;U)}^q 2^{q-1}\left(c_2^q T+\|\iota\|_{\C{L}(V;U)}^q\sum_{k=1}^{N_n}\tau^k_n\|u^{k-1+\theta}_n\|^p\right).
\end{equation}
Again, by Lemma \ref{lemma:estimates}, we obtain the desired estimate. In order to derive (\ref{eqn:estimate_der}) from (\ref{eq:theta_scheme}) we obtain
\begin{equation}
\sum_{k=1}^{N_n}\tau^k_n\left\|\frac{u^k_n-u^{k-1}_n}{\tau^k_n}\right\|_{V^*}^q \leq \sum_{k=1}^{N_n}\tau^k_n\left\|f^k_n-A^k_n u^{k-1+\theta}_n-\iota^*\xi^k_n\right\|_{V^*}^q,
\end{equation}
where $\xi^k_n\in F^k_n(\iota u^{k-1+\theta}_n)$. Moreover, we have
\begin{equation}
\sum_{k=1}^{N_n}\tau^k_n\left\|\frac{u^k_n-u^{k-1}_n}{\tau^k_n}\right\|_{V^*}^q \leq C\sum_{k=1}^{N_n}\tau^k_n(\|f^k_n\|_{V^*}^q+\|A^k_n u^{k-1+\theta}_n\|_{V^*}^q+\|\iota^*\xi^k_n\|_{V^*}^q),
\end{equation}
where $C>0$. Now (\ref{eqn:estimate_der}) follows from the estimates (\ref{eqn:estimate_ak}), (\ref{eqn:estimate_fk}), and the following inequality
$$
\sum_{k=1}^{N_n}\tau^k_n\|f^k_n\|_{V^*}^q = \sum_{k=1}^{N_n}\tau^k_n\left\|\frac{1}{\tau^k_n}\int_{t^{k-1}_n}^{t^k_n}f(t)\, dt\right\|_{V^*}^q\leq \sum_{k=1}^{N_n}\int_{t^{k-1}_n}^{t^k_n}\|f(t)\|_{V^*}^q\, dt=\|f\|_{\C{V}^*}^q.
$$
The proof is complete.
\end{proof}

\noindent We are ready to define the piecewise constant interpolants $\bar{u}_n:[0,T]\to V$ and the piecewise linear interpolants $\hat{u}_n:[0,T]\to V$ by
$$
\bar{u}_n(t)=u^{k-1+\theta}_n\,\,\ \mbox{for}\,\,\ t\in(t_n^{k-1},t_n^k]\,\,\ \ \mbox{and}\,\,\ \ \bar{u}_n(0)=u^{\theta}_n,
$$
$$
\hat{u}_n(t)=u^{k-1}_n + \frac{u^k_n-u^{k-1}_n}{\tau^k_n}(t-t^{k-1}_n)\,\,\ \mbox{for}\,\,\ t\in[t_n^{k-1},t_n^k].
$$
We also define the piecewise constant function $\bar{\eta}_n:[0,T]\to V^*$ by
\begin{equation}\label{defn:eta_bar}
\bar{\eta}_n(t)=\iota^*\xi^k_n\,\,\ \mbox{for}\,\,\ t\in(t_n^{k-1},t_n^k]\,\,\ \ \mbox{and}\,\,\ \ \bar{\eta}_n(0)=\iota^*\xi^1_n,
\end{equation}
where $\xi^k_n$ is an element of $F^k_n(\iota u^{k-1+\theta}_n)$ such that the inclusion in (\ref{eq:theta_scheme}) is realized.

\noindent We formulate the following lemma.

\begin{lemma}\label{lemma:bounded}
Under assumptions $H(A), H(F), H(U), H_0, H(t)$ and $\theta\in [\frac{1}{2},1]$, there exists $n_0\in \mathbb{N}$ such that for $n\geq n_0$, the sequence $\{\bar{u}_n\}$ is bounded in $\C{V} \cap L^\infty(0,T;H)$ and the sequence $\{\hat{u}_n\}$ is bounded in $C([0,T];H)$ with $\{\hat{u}_n'\}$ bounded in
$\C{V}^*$. Furthermore $\{\bar{u}_n\}$ is bounded in $BV^q(0,T;V^*)$. Finally $\{\pi^{q,V^*}_n\C{A}\bar{u}_n\}$ and $\{\bar{\eta}_n\}$ are bounded in $\C{V}^*$.
\end{lemma}
\begin{proof}
We choose $n_0$ such that $\tau^{max}_n < \tau_0$ for $n\geq n_0$, where $\tau_0$ is the smaller one of the tho constants appearing respectively in Lemmata \ref{lemma:estimates} and \ref{lemma:estimates2}. Such choice is possible by hypothesis $H(t)(i)$.
It suffices to show the $BV^q$ estimate since all the other estimates follow directly from Lemmata \ref{lemma:estimates} and \ref{lemma:estimates2}.
The $BV^q$ seminorm of $\bar{u}_n$ is given by
$$
\|\bar{u}_n\|_{BV^q(0,T;V^*)}^q=\sum_{j=1}^{M_n}\|u_n^{m^j_n-1+\theta}-u_n^{m^{j-1}_n-1+\theta}\|_{V^*}^q,
$$
and it is attained by the partition such that its vertices fall in the grid
intervals indexed by $m^0_n,m^1_n,\ldots,m^{M_n-1}_n,m^{M_n}_n$, where $m^0_n=1$ and $m^{M_n}_n=N_n$.
By convexity of the function $h(s)=s^q$, we obtain
$$
\|\bar{u}_n\|_{BV^q(0,T;V^*)}^q\leq\sum_{j=1}^{M_n}(m^{j-1}_n-m^j_n)^{q-1}\sum_{i=m^{j-1}_n+1}^{m^j_n}\|u_n^{i-1+\theta}-u_n^{i-2+\theta}\|_{V^*}^q\leq
$$
$$
\leq\sum_{j=1}^{M_n}(m^{j-1}_n-m^j_n)^{q-1}(\tau^{max}_n)^{q-1}\sum_{i=m^{j-1}_n+1}^{m^j_n}\tau^i_n\left\|\frac{u_n^{i-1+\theta}-u_n^{i-2+\theta}}{\tau^i_n}\right\|_{V^*}^q\leq
$$
$$
\leq (N_n\tau^{max}_n)^{q-1} \sum_{i=2}^{N_n}\tau^i_n\left\|\frac{u_n^{i-1+\theta}-u_n^{i-2+\theta}}{\tau^i_n}\right\|_{V^*}^q.
$$
Using $H(t)(ii)$, we get
$$
\|\bar{u}_n\|_{BV^q(0,T;V^*)}^q\leq (K T)^{q-1} \sum_{i=2}^{N_n}\tau^i_n\left\|\frac{\theta(u_n^{i}-u_n^{i-1})+(1-\theta)(u_n^{i-1}-u_n^{i-2})}{\tau^i_n}\right\|_{V^*}^q.
$$
We use the convexity of the function of the function $h(s)=s^q$ again to find that
$$
\|\bar{u}_n\|_{BV^q(0,T;V^*)}^q\leq (K T)^{q-1} \sum_{i=1}^{N_n}\tau^i_n\left\|\frac{u_n^{i}-u_n^{i-1}}{\tau^i_n}\right\|_{V^*}^q,
$$
which, by (\ref{eqn:estimate_der}), gives the assertion of the lemma.
\end{proof}

\noindent The next Lemma establishes weak and weak-* limits of subsequences of constructed interpolants.

\begin{lemma}\label{lemma:limit}
Under assumptions $H(A), H(F), H(U), H_0, H(t)$ and $\theta\in [\frac{1}{2},1]$, there exists $u\in \C{W}$ as well as $\zeta,\eta\in \C{V}^*$, and a subsequence of indices such that for this subsequence (still denoted by $n$), we have
\begin{eqnarray}
&&\bar{u}_n\to u \ \ \mbox{weakly in} \ \ \C{V},\\
&&\bar{u}_n\to u \ \ \mbox{weakly-* in} \ \ L^\infty(0,T;H),\\
&&\hat{u}_n\to u \ \ \mbox{weakly-* in} \ \ L^\infty(0,T;H),\label{eqn:baru}\\
&&\hat{u}'_n\to u' \ \ \mbox{weakly in} \ \ \C{V}^*,\label{eqn:baruprim}\\
&&\iota\bar{u}_n\to \iota u \ \ \mbox{strongly in} \ \ \C{U},\label{eqn:tracelimit}\\
&&\pi^{q,V^*}_n\C{A}\bar{u}_n \to \zeta \ \ \mbox{weakly in} \ \ \C{V}^*,\\
&&\bar{\eta}_n\to \eta \ \ \mbox{weakly in} \ \ \C{V}^*.
\end{eqnarray}
\end{lemma}
\begin{proof}
The fact that limits of appropriate subsequences exist follows directly from Lemmata \ref{lemma:estimates} and \ref{lemma:estimates2}.
It only suffices to prove that limits of $\hat{u}_n$ and $\bar{u}_n$ coincide. This is done in a standard way (see proof of Lemma 4 in \cite{Emmrich2009})
by showing the estimate on $\|\hat{u}_n-\bar{u}_n\|_{\C{V}^*}.$ By the direct calculation we have
$$
\|\hat{u}_n-\bar{u}_n\|_{\C{V}^*}^q = \sum_{k=1}^{N_n}\int_{t^{k-1}_n}^{t^k_n}\left\|u^{k-1+\theta}_n-u^{k-1}_n-\frac{u^k_n-u^{k-1}_n}{\tau^k_n}(t-t^{k-1}_n)\right\|_{V^*}^q\ dt=
$$
$$
=\sum_{k=1}^{N_n}\tau^k_n\left\|\frac{u^k_n-u^{k-1}_n}{\tau^k_n}\right\|_{V^*}^q \frac{(\tau^k_n)^q}{q+1} \leq \frac{(\tau^{max}_n)^q}{q+1}\sum_{k=1}^{N_n}\tau^k_n\left\|\frac{u^k_n-u^{k-1}_n}{\tau^k_n}\right\|_{V^*}^q.
$$
By the estimate (\ref{eqn:estimate_der}), it follows that $\hat{u}_n-\bar{u}_n\to 0$ in $\C{V}^*$ and therefore limits of two sequences must coincide.
\end{proof}

\begin{remark}\label{rem:ini_h}
Note that in the case of implicit Euler scheme, i.e. $\theta=1$, the lemmata \ref{lemma:estimates}, \ref{lemma:estimates2}, \ref{lemma:bounded}, and \ref{lemma:limit} remain valid if elements of the sequence $\{u_{0n}\}$, that approximates the initial condition $u_0$, belong to $H$ and not necessarily to $V$.  
\end{remark}

\begin{theorem} \label{thm:main}
Under assumptions $H(A), H(F), H(U), H_0, H(t)$ and $\theta\in [\frac{1}{2},1]$, the function $u$ obtained in Lemma \ref{lemma:limit}
solves Problem $\mathbf{(\C{P})}$.
\end{theorem}
\begin{proof}
First we show that $u$ satisfies the initial condition. From (\ref{eqn:baru}) and (\ref{eqn:baruprim}), it follows, by Corollary 4 of \cite{Simon1987}, that
\begin{equation}\label{eqn:strongconv}
\hat{u}_n\to u \ \ \mbox{strongly in} \ \ C([0,T];V^*),
\end{equation}
and furthermore $\hat{u}_n(0)\to u(0)$ strongly in $V^*$. Since $\hat{u}_n(0)=u_{0n}$ and $u_{0n}\to u_0$ strongly in $H$, from the uniqueness of the limit,
it follows that $u(0)=u_0$.

Now let us observe that from (\ref{eq:theta_scheme}), the following equality holds in $\C{V}^*$ for $k\in \{1,\ldots,N_n\}$
\begin{equation}\label{eqn:theta_scheme}
\hat{u}_n'+\pi^{q,V^*}_n\C{A}\bar{u}_n+\bar{\eta}_n=\pi^{q,V^*}_{n}f.
\end{equation}
Note that $\pi^{q,V^*}_{n} f\to f$ strongly in $\C{V}^*$. Therefore we can pass to the limit in
(\ref{eqn:theta_scheme}) and find
\begin{equation}\label{eqn:actual}
u'+\zeta+\eta=f.
\end{equation}
To conclude the proof we must verify that $\zeta=\C{A}u$ and $\eta(t)\in \iota^*F(t,\iota u(t))$ for a.e. $t$.

Next, we verify that $\eta(t)\in \iota^*F(t,\iota u(t))$ for a.e. $t$. We remind that $\bar{\eta}_n(t)=\iota^* \xi^k_n$ for $t\in (t^{k-1}_n,t^k_n)$ and $\xi^k_n=\frac{1}{\tau^k_n}\int^{t^{k-1}_n}_{t^k_n}\xi_n(t)\ dt$ with $\xi_n(t)\in F(t,\iota u_n^{k-1+\theta})$ for a.e. $t\in (t^{k-1}_n,t^k_n)$. Let us denote $\eta_n=\pi_n^{q,U^*}\xi_n$. For a.e. $t$ there holds $\bar{\eta}_n(t)=\iota^*\eta_n(t)$. We will show that
\begin{equation}\label{eqn:approxconverg}
\xi_n-\eta_n \to 0\ \ \mbox{weakly in}\ \ \C{U}^*.
\end{equation}
To this end, for $w\in \C{U}$ we have
\begin{eqnarray}
&&\int_0^T\langle \xi_n(t)-\eta_n(t),w(t)\rangle_{U^*\times U}\ dt = \nonumber\\
&&=  \sum_{k=1}^{N_n}\int_{t_n^{k-1}}^{t_n^k}\langle \xi_n(t)-\frac{1}{\tau^k_n}\int^{t^{k-1}_n}_{t^k_n}\xi_n(s)\ ds,w(t)\rangle_{U^*\times U}\ dt=\nonumber\\
&& = \sum_{k=1}^{N_n}\frac{1}{\tau^k_n}\int_{t_n^{k-1}}^{t_n^k}\int_{t_n^{k-1}}^{t_n^k}\langle \xi_n(t)-\xi_n(s),w(t)\rangle_{U^*\times U}\ ds\ dt.\nonumber
\end{eqnarray}
Analogously as in the proof of Lemma \ref{lemma:estimates2}, by the estimates (\ref{eqn:casea}) and (\ref{eqn:caseb}), the sequence $\xi_n$ is bounded
in $\C{U}^*$, so, for a subsequence, we may assume that $\xi_n\to \xi$ weakly in this space, where $\xi\in \C{U}^*$. We obtain
\begin{eqnarray}
&&\int_0^T\langle \xi_n(t)-\eta_n(t),w(t)\rangle_{U^*\times U}\ dt =\label{eqn:hemitozero}\\
&&=\sum_{k=1}^{N_n}\frac{1}{\tau^k_n}\int_{t_n^{k-1}}^{t_n^k}\int_{t_n^{k-1}}^{t_n^k}\langle \xi_n(t)-\xi(t),w(t)\rangle_{U^*\times U}\ ds\ dt + \nonumber\\
&&+\sum_{k=1}^{N_n}\frac{1}{\tau^k_n}\int_{t_n^{k-1}}^{t_n^k}\int_{t_n^{k-1}}^{t_n^k}\langle\xi(t)-\xi(s),w(t)\rangle_{U^*\times U}\ ds\ dt+\nonumber\\
&&+ \sum_{k=1}^{N_n}\frac{1}{\tau^k_n}\int_{t_n^{k-1}}^{t_n^k}\int_{t_n^{k-1}}^{t_n^k}\langle\xi(s)-\xi_n(s),w(t)\rangle_{U^*\times U}\ ds\ dt:=I_1+I_2+I_3.\nonumber
\end{eqnarray}
We consider three terms separately.
$$
I_1 = \sum_{k=1}^{N_n}\int_{t_n^{k-1}}^{t_n^k}\langle \xi_n(t)-\xi(t),w(t)\rangle_{U^*\times U}\ dt = \langle\xi_n-\xi,w\rangle_{\C{U}^*\times\C{U}}\to 0
$$
$$
I_2=\sum_{k=1}^{N_n}\int_{t_n^{k-1}}^{t_n^k}\langle\xi(t)-\frac{1}{\tau^k_n}\int_{t_n^{k-1}}^{t_n^k}\xi(s)\ ds,w(t)\rangle_{U^*\times U}\ ds\ dt =
\langle\xi-\pi^{q,U^*}_n \xi, w \rangle_{\C{U}^*\times\C{U}}.
$$
Since $\pi^{q,U^*}_n \xi \to \xi$ strongly in $\C{U}^*$, we conclude that $I_2\to 0$, as $n\to \infty$.
$$
I_3 = \sum_{k=1}^{N_n}\int_{t_n^{k-1}}^{t_n^k}\langle \xi_n(t)-\xi(t),\frac{1}{\tau^k_n} \int_{t_n^{k-1}}^{t_n^k}w(s)\ ds\rangle_{U^*\times U}\ dt = \langle\xi_n-\xi,\pi^{p,U}_n w \rangle_{\C{U}^*\times\C{U}}.
$$
Since $\pi^{p,U}_n w \to w$ strongly in $\C{U}$, we conclude that $I_3\to 0$, as $n\to \infty$. Now we have
$$
\iota^*\xi_n-\iota^*\eta_n \to 0\ \ \mbox{weakly in}\ \ \C{V}^*,
$$
which means that
$$
\iota^*\xi_n \to \eta\ \ \mbox{weakly in}\ \ \C{V}^*,
$$
and moreover $\eta = \iota^*\xi$. Since $\xi_n\to \xi$ weakly in $\C{U}^*$, $\iota\bar{u}_n\to\iota u$ strongly in $\C{U}$ and we have   $\xi_n(t)\in F(t,\iota \bar{u}_n(t))$ for a.e. $t\in (0,T)$, from the argument based on Aubin and Cellina theorem (see Proposition 2 in \cite{Migorski2009}), we conclude that $\xi(t)\in F(t,\iota u(t))$ for a.e. $t\in (0,T)$ and the result is shown.

Subsequently we verify that $\zeta=\C{A}u$. We show that
\begin{equation}\label{eqn:auxiliary}
\pi^{q,V^*}_n \C{A} \bar{u}_n - \C{A} \bar{u}_n \to 0 \ \ \mbox{weakly in}\ \ \C{V}^*.
\end{equation}
This is done similarly as in the proof of (\ref{eqn:approxconverg}). For all $w\in \C{V}$ we have
$$
\langle\pi^{q,V^*}_n \C{A} \bar{u}_n - \C{A} \bar{u}_n, w\rangle_{\C{V}^*\times \C{V}} =
$$
$$
=\sum_{k=1}^{N_n}\frac{1}{\tau_k^n}\int_{t_n^{k-1}}^{t_n^k}\int_{t_n^{k-1}}^{t_n^k}\langle A(s,\bar{u}_n(s)) - A(t,\bar{u}_n(t)),w(t)\rangle\ dt\ ds.
$$
Analogously, as in the proof of (\ref{eqn:estimate_ak}) in Lemma \ref{lemma:estimates2}, we obtain that $\C{A}\bar{u}_n$ is bounded in $\C{V}^*$ so we may extract a subsequence such that $\C{A}\bar{u}_n\to \lambda$ weakly in $\C{V}^*$ with $\lambda\in\C{V}^*$. Therefore we have
$$
\langle\pi^{q,V^*}_n \C{A} \bar{u}_n - \C{A} \bar{u}_n, w\rangle_{\C{V}^*\times \C{V}} =
$$
$$
=\sum_{k=1}^{N_n}\frac{1}{\tau_k^n}\int_{t_n^{k-1}}^{t_n^k}\int_{t_n^{k-1}}^{t_n^k}\langle A(s,\bar{u}_n(s)) - \lambda(s),w(t)\rangle\ dt\ ds+
$$
$$
+\sum_{k=1}^{N_n}\frac{1}{\tau_k^n}\int_{t_n^{k-1}}^{t_n^k}\int_{t_n^{k-1}}^{t_n^k}\langle \lambda(s)-\lambda(t),w(t)\rangle\ dt\ ds+
$$
$$
+\sum_{k=1}^{N_n}\frac{1}{\tau_k^n}\int_{t_n^{k-1}}^{t_n^k}\int_{t_n^{k-1}}^{t_n^k}\langle \lambda(t) - A(t,\bar{u}_n(t)) ,w(t)\rangle\ dt\ ds.
$$
Above expression tends to zero, by the same argument as in (\ref{eqn:hemitozero}). Note that we have also shown that
\begin{equation}\label{eqn:zeta}
\C{A}\bar{u}_n\to \zeta\ \ \mbox{weakly in}\ \ \C{V}^*,
\end{equation}
i.e. $\lambda=\zeta$. Now let us observe that for all $n\in\mathbb{N}$, we have
$
\langle \pi^{q,V^*}_n \C{A} \bar{u}_n - \C{A} \bar{u}_n, \bar{u}_n \rangle_{\C{V}^*\times \C{V}} = 0.
$
This follows from the following computation
$$
\langle \pi^{q,V^*}_n \C{A} \bar{u}_n - \C{A} \bar{u}_n, \bar{u}_n \rangle_{\C{V}^*\times \C{V}} =
$$
$$
=\sum_{k=1}^{N_n} \frac{1}{\tau^k_n}\int_{t^{k-1}_n}^{t^k_n}\int_{t^{k-1}_n}^{t^k_n}\langle A(s,u^{k-1+\theta}_n) - A(t,u^{k-1+\theta}_n),u^{k-1+\theta}_n\rangle\ ds\ dt=
$$
$$
= \sum_{k=1}^{N_n} \int_{t^{k-1}_n}^{t^k_n}\langle A(s,u^{k-1+\theta}_n) ,u^{k-1+\theta}_n\rangle\ ds -
$$
$$
- \sum_{k=1}^{N_n} \int_{t^{k-1}_n}^{t^k_n}\langle  A(t,u^{k-1+\theta}_n),u^{k-1+\theta}_n\rangle\ dt = 0.
$$
Thus, we have
\begin{equation} \label{eqn:int_nemytskii}
\lim_{n\to \infty}\langle\pi^{q,V^*}_n \C{A} \bar{u}_n - \C{A} \bar{u}_n,\bar{u}_n - u\rangle_{\C{V}^*\times \C{V}} = 0.
\end{equation}
We observe that
$$
\lim_{n\to\infty} \langle\pi^{q,V^*}_n f, \bar{u}_n - u\rangle_{\C{V}^*\times \C{V}} = 0,
$$
$$
\lim_{n\to\infty} \langle\bar{\eta}_n, \bar{u}_n - u\rangle_{\C{V}^*\times \C{V}} = \lim_{n\to\infty} \langle\eta_n, \iota \bar{u}_n - \iota u\rangle_{\C{U}^*\times \C{U}}=0,
$$
where the last limit follows from (\ref{eqn:tracelimit}) and the fact that $\eta_n\to \xi$ weakly in $\C{U}^*$. Now using the equalities
$$
\limsup_{n\to\infty}\langle \C{A} \bar{u}_n,\bar{u}_n - u\rangle_{\C{V}^*\times \C{V}} = \limsup_{n\to\infty}\langle \pi^{q,V^*}_n \C{A} \bar{u}_n,\bar{u}_n - u\rangle_{\C{V}^*\times \C{V}} =
$$
$$
 = \limsup_{n\to\infty}\langle \pi^{q,V^*}_n f - \hat{u}_n' - \bar{\eta}_n,\bar{u}_n - u\rangle_{\C{V}^*\times \C{V}},
$$
we obtain
$$
\limsup_{n\to\infty}\langle \C{A} \bar{u}_n,\bar{u}_n - u\rangle_{\C{V}^*\times \C{V}}
 = \limsup_{n\to\infty}\langle  \hat{u}_n',u-\bar{u}_n\rangle_{\C{V}^*\times \C{V}}.
$$
Now we observe that
$$
\langle\hat{u}_n',u-\bar{u}_n\rangle_{\C{V}^*\times \C{V}} = (\hat{u}_n',u-\bar{u}_n)_{\C{H}\times\C{H}} = (\hat{u}_n',\hat{u}_n-\bar{u}_n)_{\C{H}\times\C{H}} +
$$
$$
+ \frac{1}{2}(\|u_{n0}\|_H^2-\|\hat{u}_n(T)\|_H^2)+ \langle \hat{u}_n',u\rangle_{\C{V}^*\times \C{V}}
$$
and
\begin{equation}\label{eqn:BBB}
(\hat{u}_n',\hat{u}_n-\bar{u}_n)_{\C{H}\times\C{H}} = - \frac{2\theta-1}{2}\sum_{k=1}^{N_n}\|u^k_n-u^{k-1}_n\|_H^2 \leq 0.
\end{equation}
Hence
$$
\limsup_{n\to\infty}\langle\hat{u}_n',u-\bar{u}_n\rangle_{\C{V}^*\times \C{V}} \leq \frac{1}{2}(\|u_0\|_H^2-\liminf_{n\to\infty}\|\hat{u}_n(T)\|_H^2) + \langle u',u\rangle_{\C{V}^*\times \C{V}}.
$$
Finally, we observe that $\|\hat{u}_n(T)\|_H$ is bounded and therefore we may assume that for a subsequence $\hat{u}_n(T) \to w$ weakly in $H$ with $w\in H$. It follows from (\ref{eqn:strongconv}) that $\hat{u}_n(T) \to u(T)$ strongly in $V^*$. We conclude that $w=u(T)$ and the convergence holds for the whole subsequence for which the assertion of Lemma \ref{lemma:limit} holds. From the weak lower semicontinuity of norm, we have
$$
\limsup_{n\to\infty}\langle\hat{u}_n',u-\bar{u}_n\rangle_{\C{V}^*\times \C{V}} \leq \frac{1}{2}(\|u_0\|_H^2-\|u(T)\|_H^2) + \langle u',u\rangle_{\C{V}^*\times \C{V}}=0,
$$
which gives
\begin{equation}\label{eqn:AA}
\limsup_{n\to\infty}\langle \C{A} \bar{u}_n,\bar{u}_n - u\rangle_{\C{V}^*\times \C{V}}
\leq 0.
\end{equation}
Using the pseudomonotonicity of the Nemytskii operator (Lemma \ref{lemma:pseudo:nemytskii}) we conclude that  $\langle \C{A}u, u-y\rangle_{\C{V}^*\times\C{V}} \leq \liminf_{n\to\infty}\langle \C{A}\bar{u}_n, \bar{u}_n-y\rangle_{\C{V}^*\times\C{V}}$ for all $y\in\C{V}$. The assertion follows by (\ref{eqn:zeta}) taking respectively $y=u+w$ and $y=u-w$, where $w\in \C{V}^*$.
\end{proof}

\begin{remark}Note that if $\theta=1$ then the Theorem 
\ref{thm:main} remains valid if $\{u_{0n}\}\subset H$ and not necessarily $\{u_{0n}\}\subset V$. In this case the piecewise linear interpolant $\hat{u}_n$ does not assume values in $V$, but only in $H$ near the starting time point $t=0$.
\end{remark}

\section{Strong convergence results}\label{sec:improved}

In this section we provide two more results on the convergence of approximate solutions. First we show that piecewise constant interpolants converge
strongly in $\C{V}$ provided $A(t,\cdot)$ are of type $(S)_+$. Then we show, following
\cite{Emmrich2009}, that under the restriction on the time grids, piecewise linear interpolants converge weakly in $\C{V}$ and if additionally
$A$ is H\"{o}lder continuous with respect to time, then they converge pointwise strongly in $H$.

\begin{theorem}
If, in addition to assumptions $H(A), H(F), H(U), H_0$ and $H(t)$, the operator $A(t,\cdot)$ is of type $(S)_+$ for a.e. $t\in (0,T)$, and if $\theta\in [\frac{1}{2},1]$, then the convergence $\bar{u}_n\to u$ holds in the strong topology of $\C{V}$.
\end{theorem}
\begin{proof} In the proof of Theorem \ref{thm:main} we have shown that
$$
\limsup_{n\to\infty}\langle \C{A} \bar{u}_n,\bar{u}_n - u\rangle_{\C{V}^*\times \C{V}}
\leq 0,
$$
cf (\ref{eqn:AA}). Since $\bar{u}_n\to u$ weakly in $\C{V}$ and the sequence $\bar{u}_n$ is bounded in $M^{p,q}(0,T;V,V^*)$ the conclusion follows easily from Lemma \ref{lemma:splus}.
\end{proof}

\noindent We remark that there is no convergence of piecewise linear interpolants $\hat{u}_n\to u$ in the weak topology of $\C{V}$ unless we impose the restrictive assumptions on the time grid (see \cite{Emmrich2009}). Under these assumptions, it is also possible to show that $\hat{u}_n(t)\to u(t)$ strongly in $H$ for all $t\in [0,T]$ (note that this convergence for a.e. $t\in [0,T]$ follows from Lions-Aubin lemma). To this end, let $r^k_n=\frac{\tau^k_n}{\tau^{k-1}_n}$ for $k\in \{2,\ldots,N_n\}$ and $r^{max}_n=\max_{k\in \{2,\ldots,N_n\}}r^k_n$. We formulate the following result

\begin{theorem}\label{thm:improved}
Under assumptions $H(A), H(F), H(U), H_0, H(t)$,  $\theta\in [\frac{1}{2},1]$,
\begin{equation}\label{eqn:initoinfty}
\|u_{0n}\|_V\leq \frac{C}{\sqrt[p]{\tau^{max}_n}},
\end{equation}
and if there exists the constant $R>0$ such that
\begin{equation}\label{eqn:timedec}
r^{max}_n\leq R < \left(\frac{\theta}{1-\theta}\right)^p\ \mbox{for all}\  n\in \mathbb{N},
\end{equation}
(for $\theta=1$ the constant $R$ can be chosen arbitrary) then for the convergent subsequence established by Lemma \ref{lemma:limit},
we have $\hat{u}_n\to u$ weakly in $\C{V}$. Moreover if the nonlinear operator $A$
is H\"{o}lder continuous with respect to time, in the sense that for all $s,t\in [0,T]$ and for all $v\in V$ we have
$$
\|A(t,v)-A(s,v)\|_{V^*}\leq (C_1+C_2\|v\|^\delta)|t-s|^\gamma,
$$
where $C_1,C_2>0$, $\gamma\in (0,1]$ and $\delta\in (0,p(\gamma+1)-1)$, then
$\hat{u}_n(t)\to u(t)$ strongly in $H$ for all $t\in [0,T]$.
\end{theorem}
\begin{proof}
\noindent First we prove that $\hat{u}_n\to u$ weakly in $\C{V}$. Following the lines of the proof of Lemma 3 in \cite{Emmrich2009}, starting from the inequality
$$
\|u^k_n\|\leq\frac{1}{\theta}\|u^{k-1+\theta}_n\| + \frac{1-\theta}{\theta}\|u^{k-1}_{n}\|
$$
valid for all $n\in \mathbb{N}$ and $k\in\{1,\ldots,N_n\}$, by a direct computation which uses the Minkowski inequality we get
\begin{eqnarray}
&&\left(\sum_{k=1}^{N_n}\tau^k_n\|u^k_n\|^p\right)^{\frac{1}{p}}\leq \frac{1}{\theta}\left(\sum_{k=1}^{N_n}\tau^k_n\|u^{k-1+\theta}_n\|^p\right)^{\frac{1}{p}} + \\
&&+\frac{1-\theta}{\theta}\sqrt[p]{\tau^{max}_n}\|u_{0n}\| + \frac{1-\theta}{\theta}\sqrt[p]{R}\left(\sum_{k=1}^{N_n}\tau^k_n\|u^k_n\|^p\right)^{\frac{1}{p}}.\nonumber
\end{eqnarray}
By Lemma \ref{lemma:estimates} and the hypothesis (\ref{eqn:initoinfty}) we obtain
$$
\left(\sum_{k=1}^{N_n}\tau^k_n\|u^k_n\|^p\right)^{\frac{1}{p}}\leq M + \frac{1-\theta}{\theta}\sqrt[p]{R}\left(\sum_{k=1}^{N_n}\tau^k_n\|u^k_n\|^p\right)^{\frac{1}{p}},
$$
with $M_1>0$. By (\ref{eqn:timedec}) we have the boundedness of $\left(\sum_{k=1}^{N_n}\tau^k_n\|u^k_n\|^p\right)^{\frac{1}{p}}$.

In order to show the weak convergence $\hat{u}_n\to u$ in $\C{V}$ it is enough to obtain the bound in $\C{V}$ of the considered sequence. We have
$$
\|\hat{u}_n\|_{\C{V}}^p = \sum_{k=1}^{N_n}\int_{t^{k-1}_n}^{t^k_n}\left\|u_n^{k-1+\theta}+\frac{u^k_n-u^{k-1}_n}{\tau^k_n}(t-t^{k-1}_n-\theta\tau^k_n)\right\|^p\ dt.
$$
By a direct calculation, we arrive at
$$
\|\hat{u}_n\|_{\C{V}}^p \leq 2^{p-1}\sum_{k=1}^{N_n}\tau^k_n\|u^{k-1+\theta}_n\|^p +\frac{2^{p-1}(\theta^{p+1}+(1-\theta)^{p+1})}{p+1} \sum_{k=1}^{N_n}\tau^k_n\|u^k_n-u^{k-1}_n\|^p.
$$
We only need to show that the last sum is bounded. This holds due to the following estimate
$$
\sum_{k=1}^{N_n}\tau^k_n\|u^k_n-u^{k-1}_n\|^p \leq 2^{p-1}\sum_{k=1}^{N_n}\tau^k_n\|u^k_n\|^p + 2^{p-1}\sum_{k=1}^{N_n}\tau^k_n\|u^{k-1}_n\|^p \leq
$$
$$
\leq 2^{p-1}(1+R)\sum_{k=1}^{N_n}\tau^k_n\|u^k_n\|^p + 2^{p-1}\tau^{max}_n\|u_{n0}\|^p.
$$
In what follows, we show that $\hat{u}_n(t)\to u(t)$ strongly in $H$ for all $t\in [0,T]$. Subtracting the equation (\ref{eqn:actual}) from (\ref{eqn:theta_scheme}) and taking the duality with $\bar{u}_n-u$, we obtain
\begin{eqnarray}\label{eq:time_t}
&&\langle \hat{u}_n'-u',\bar{u}_n-u \rangle_{L^q(0,t;V^*)\times L^p(0,t;V)}+\label{eqn:oszacowania}\\
&&+\langle \pi_n^{q,V^*}\C{A}\bar{u}_n-\C{A}u, \bar{u}_n-u \rangle_{L^q(0,t;V^*)\times L^p(0,t;V)}+\nonumber\\
&&+\langle \bar{\eta}_n-\xi, \iota\bar{u}^n-\iota u \rangle_{L^q(0,t;U^*)\times L^p(0,t;U)} =\nonumber\\
&&= \langle \pi_n^{q,V^*}f - f, \bar{u}_n-u \rangle_{L^q(0,t;V^*)\times L^p(0,t;V)},\nonumber
\end{eqnarray}
where $\xi(s) \in F(s,\iota u(s))$ for a.e. $s\in (0,T)$ and $\bar{\eta}_n$ is given by (\ref{defn:eta_bar}). We pass with $n$ to infinity. Since $\pi_n^{q,V^*}f \to f$
strongly in $\C{V}^*$ and moreover in $L^q(0,t;V^*)$ and $\bar{u}_n \to u$ weakly in $\C{V}$ and moreover in $L^p(0,t;V)$ we have
\begin{equation}\label{eqn:oszacowania2}
\lim_{n\to \infty}\langle \pi_n^{q,V^*}f - f, \bar{u}_n-u \rangle_{L^q(0,t;V^*)\times L^p(0,t;V)} = 0.
\end{equation}
Now since $\iota\bar{u}^n\to \iota u$ strongly in $\C{U}$ and moreover in $L^p(0,t;U)$ and from growth condition and estimate on $\bar{u}_n$ the sequence $\bar{\eta}_n$ is bounded in $\C{U}^*$ and moreover in $L^q(0,t;U^*)$, we get
\begin{equation}\label{eqn:oszacowania3}
\lim_{n\to \infty}\langle \eta_n-\xi, \iota\bar{u}^n-\iota u \rangle_{L^q(0,t;U^*)\times L^p(0,t;U)} = 0.
\end{equation}
Using (\ref{eqn:oszacowania2}) and (\ref{eqn:oszacowania3}) in (\ref{eqn:oszacowania}), we have
\begin{equation}\label{eqn:1}
\lim_{n\to\infty}\langle \hat{u}_n'-u'+\pi_n^{q,V^*}\C{A}\bar{u}_n,\bar{u}_n-u \rangle_{L^q(0,t;V^*)\times L^p(0,t;V)} = 0.
\end{equation}

We show the following result
\begin{equation}\label{eqn:2}
\lim_{n\to\infty} \langle \pi_n^{q,V^*}\C{A}\bar{u}_n-\C{A}\bar{u}_n,\bar{u}_n-u \rangle_{L^q(0,t;V^*)\times L^p(0,t;V)} = 0.
\end{equation}
From (\ref{eqn:auxiliary}), it follows that $\pi_n^{q,V^*}\C{A}\bar{u}_n-\C{A}\bar{u}_n\to 0$ weakly in $L^q(0,t;V^*)$, so it is enough to
show that
$$
\lim_{n\to\infty} \langle \pi_n^{q,V^*}\C{A}\bar{u}_n-\C{A}\bar{u}_n,\bar{u}_n \rangle_{L^q(0,t;V^*)\times L^p(0,t;V)} = 0.
$$
To this end, let us denote by $m$ the index of the largest point of the $n$-th grid which is less than $t$. We have
$$
\langle \pi_n^{q,V^*}\C{A}\bar{u}_n-\C{A}\bar{u}_n,\bar{u}_n \rangle_{L^q(0,t;V^*)\times L^p(0,t;V)} =
$$
$$
=\sum_{k=1}^{m}\int_{t^{k-1}_n}^{t^k_n}\left\langle \frac{1}{\tau^k_n}\int_{t^{k-1}_n}^{t^k_n}A(s,u^{k-1+\theta})\ ds-A(r,u^{k-1+\theta}),u^{k-1+\theta}_n \right\rangle\ dr +
$$
$$
+\int_{t^{m}_n}^{t}\left\langle \frac{1}{\tau^{m+1}_n}\int_{t^{m}_n}^{t^{m+1}_n}A(s,u^{m+\theta})\ ds-A(r,u^{m+\theta}),u^{m+\theta}_n \right\rangle\ dr.
$$
The first term in the right-hand side of above relation is equal to zero analogously to the proof of (\ref{eqn:int_nemytskii}). We estimate the second term. Using the
H\"{o}lder continuity of $A$, we obtain
$$
\left|\int_{t^{m}_n}^{t}\left\langle \frac{1}{\tau^{m+1}_n}\int_{t^{m}_n}^{t^{m+1}_n}A(s,u^{m+\theta})\ ds-A(r,u^{m+\theta}),u^{m+\theta}_n \right\rangle\ dr\right| \leq
$$
$$
\leq \frac{1}{\tau^{m+1}_n} \left|\int_{t^{m}_n}^{t}\int_{t^{m}_n}^{t^{m+1}_n}\| A(s,u^{m+\theta})-A(r,u^{m+\theta})\|_{V^*}\|u^{m+\theta}_n\|\ ds \ dr\right| \leq
$$
$$
\leq \frac{2 (\tau^{m+1}_n)^{\gamma+1}}{(\gamma+1)(\gamma+2)}(C_1+C_2\|u^{m+\theta}_n\|^\delta)\|u^{m+\theta}_n\| \leq
$$
$$
\leq (\tau^{m+1}_n)^{\gamma+1}(C_3+C_4\|u^{m+\theta}_n\|^{\delta+1}),
$$
where $C_3$ and $C_4$ are positive constants. Since $\tau^{m+1}_n$ tends to zero as $n\to \infty$, the term $C_3(\tau^{m+1}_n)^{\gamma+1}$ also tends to zero. The remaining term can be represented as
$$
C_4 (\tau^{m+1}_n)^{\gamma+1}\|u^{m+\theta}_n\|^{\delta+1} = C_4 (\tau^{m+1}_n)^{\gamma+1 - \frac{\delta+1}{p}}(\tau^{m+1}_n\|u^{m+\theta}_n\|^p)^{\frac{\delta+1}{p}}.
$$
The expression $\tau^{m+1}_n\|u^{m+\theta}_n\|^p$ is bounded from Lemma \ref{lemma:estimates} and, since
$\gamma+1 - \frac{\delta+1}{p} > 0$, we get $(\tau^{k(n)+1}_n)^{\gamma+1 - \frac{\delta+1}{p}} \to 0$ as $n\to\infty$ .
Having shown (\ref{eqn:2}), we deduce from (\ref{eqn:1}) that
\begin{equation}\label{eqn:3}
\lim_{n\to\infty}\langle \hat{u}_n'-u'+\C{A}\bar{u}_n,\bar{u}_n-u \rangle_{L^q(0,t;V^*)\times L^p(0,t;V)} = 0.
\end{equation}

Now we need to show that
$$
\limsup_{n\to\infty}\langle \C{A} \bar{u}_n,\bar{u}_n - u\rangle_{L^q(0,t;V^*)\times L^p(0,t;V)}\leq 0.
$$
This follows analogously as the proof of (\ref{eqn:AA}) in Theorem \ref{thm:main}. The only delicate step in the proof is showing
$(\hat{u}_n',\hat{u}_n-\bar{u}_n)_{L^2(0,t;H)\times L^2(0,t;H)}\leq 0$. This follows from the calculation
\begin{eqnarray}
&&(\hat{u}_n',\hat{u}_n-\bar{u}_n)_{L^2(0,t;H)\times L^2(0,t;H)} = - \frac{2\theta-1}{2}\sum_{k=1}^{m}\|u^k_n-u^{k-1}_n\|_H^2 +\nonumber\\
&&+\int_{t^{m}_n}^t\left\|\frac{u^{k+1}_n-u^k_n}{\tau^{m+1}_n}\right\|_H^2(t-t^{m}_n-\theta\tau^{m+1}_n)\ dt\leq 0,\label{eqn:5}
\end{eqnarray}
where the inequality holds since the value of the integral is nonpositive for $t\in (t^{k(n)}_n, t^{k(n)+1}_n)$.
Now, from Lemma \ref{lemma:pseudo:nemytskii}, it follows that
$$
0\leq \liminf_{n\to\infty}\langle \C{A} \bar{u}_n,\bar{u}_n - u\rangle_{L^q(0,t;V^*)\times L^p(0,t;V)},
$$
so
$$
\langle \C{A} \bar{u}_n,\bar{u}_n - u\rangle_{L^q(0,t;V^*)\times L^p(0,t;V)} \to 0.
$$
Thus (\ref{eqn:3}) implies
\begin{equation}\label{eq:derivative_t}
\lim_{n\to\infty}\langle \hat{u}_n'-u',\bar{u}_n-u \rangle_{L^q(0,t;V^*)\times L^p(0,t;V)} = 0.
\end{equation}
The last equation can be reformulated as
$$
\lim_{n\to\infty}\left(\frac{1}{2}(\|\hat{u}_n(t)-u(t)\|_H^2 - \|u_{n0}-u_0\|_H^2 )+ \right.
$$
$$\left.+( \hat{u}_n',\bar{u}_n-\hat{u}_n )_{L^2(0,t;H)\times L^2(0,t;H)}-\langle u',\bar{u}_n-\hat{u}_n \rangle_{L^q(0,t;V^*)\times L^p(0,t;V)}\right) = 0.
$$
Since $u_{n0}\to u_0$ strongly in $H$ and $\bar{u}_n-\hat{u}_n\to 0$ weakly in $\C{V}$ and also weakly in $L^p(0,t;V)$ we can write
$$
0 = \lim_{n\to\infty}\left(\frac{1}{2}\|\hat{u}_n(t)-u(t)\|_H^2+ ( \hat{u}_n',\bar{u}_n-\hat{u}_n )_{L^2(0,t;H)\times L^2(0,t;H)}\right).
$$
We apply (\ref{eqn:5}) and get
$$
0 \geq \frac{1}{2}\limsup_{n\to\infty}\|\hat{u}_n(t)-u(t)\|_H^2,
$$
and it follows that $\|\hat{u}_n(t)-u(t)\|_H\to 0$ for all $t\in [0,T]$, which concludes the proof.
\end{proof}

We consider seperately the improved convergence in the case when $\theta=1$ and $\{u_{n0}\}\subset H$. We formulate the following Theorem
\begin{theorem}\label{thm:improved}
Let $\varepsilon>0$ and $\{u_{n0}\}\subset H$ be such that $u_{n0}\to u_0$ strongly in $H$. Under assumptions $H(A), H(F), H(U), H_0, H(t)$,  $\theta=1$,
and if there exists the constant $R>0$ such that
\begin{equation}\label{eqn:timedec}
r^{max}_n\leq R,
\end{equation}
then for the convergent subsequence established by Lemma \ref{lemma:limit},
we have $\hat{u}_n\to u$ weakly in $L^p(\varepsilon,T;V)$. Moreover if the nonlinear operator $A$
is H\"{o}lder continuous with respect to time, in the sense that for all $s,t\in [0,T]$ and for all $v\in V$ we have
$$
\|A(t,v)-A(s,v)\|_{V^*}\leq (C_1+C_2\|v\|^\delta)|t-s|^\gamma,
$$
where $C_1,C_2>0$, $\gamma\in (0,1]$ and $\delta\in (0,p(\gamma+1)-1)$, then
$\hat{u}_n(t)\to u(t)$ strongly in $H$ for all $t\in [\varepsilon,T]$.
\end{theorem}
\begin{proof}
Let us first estimate
$$
\int_{t^{k-1}_n}^{t^k_n}\|\bar{u}_n(t)-\hat{u}_n(t)\|^p\ dt = \frac{\tau^k_n}{p+1}\|u^k_{n}-u^{k-1}_n\|^p\leq \frac{2^{p-1}}{p+1}(\tau^k_n \|u^k_n\|^p+ R\tau^{k-1}_n\|u^{k-1}_n\|^p).
$$
Fix $\varepsilon>0$. Let $K(n,\varepsilon)$ be the smallest index such that $t^{K(n,\varepsilon)}_n>\varepsilon$. We have
$$
\|\bar{u}_n-\hat{u}_n\|^p_{L^p(\varepsilon,T;V)}\leq \sum_{i=K(n,\varepsilon)}^{N_n} \frac{2^{p-1}}{p+1}(\tau^k_n \|u^k_n\|^p+ R\tau^{k-1}_n\|u^{k-1}_n\|^p) \leq 
$$
$$
\leq \frac{2^{p-1}(1+R)}{p+1} \sum_{i=K(n,\varepsilon)-1}^{N_n} \tau^k_n \|u^k_n\|^p.
$$
Since $\tau^{max}_n\to 0$ for large enough $n$ we have 
$K(n,\varepsilon)-1 \geq 1$ and, from lemmata  \ref{lemma:estimates} and \ref{lemma:bounded} (see also Remark \ref{rem:ini_h}) we obtain that $\hat{u}_n$ is bounded in $L^p(\varepsilon,T;V)$. It follows that, for a subsequence, $\hat{u}_n\to u$ weakly in this space.
Since, by Lemma \ref{lemma:bounded}, $\hat{u}_n'$ is bounded in $\C{V}^*$ and in $L^q(\varepsilon,T;V^*)$, then from the Lions-Aubin compactness lemma it follows that $\hat{u}_n\to u$ strongly in $L^p(\varepsilon,T;H)$. From arbitrariety of $\varepsilon$ it follows that, for another subsequence, $\hat{u}_n(t)\to u(t)$ strongly in $H$ for a.e. $t\in (0,T)$. To prove the strong convergence for all $t\in [0,T]$ pick $t>0$. We are able to find $\varepsilon\in (0,t)$ such that $\hat{u}_n(\varepsilon)\to u(\varepsilon)$ strongly in $H$. We proceed analogously to the argument in the proof of Theorem \ref{thm:improved} and from the analogue of (\ref{eq:time_t}) with $(0,t)$ replaced with $(\varepsilon,t)$, using the H\"{o}lder continuity of $A$, we obtain the following analogue of (\ref{eq:derivative_t})
\begin{equation}
\lim_{n\to\infty}\langle \hat{u}_n'-u',\bar{u}_n-u \rangle_{L^q(\varepsilon,t;V^*)\times L^p(\varepsilon,t;V)} = 0.
\end{equation}
This means that
$$
\lim_{n\to\infty}\left(\frac{1}{2}(\|\hat{u}_n(t)-u(t)\|_H^2 - \|\hat{u}_n(\varepsilon)-u(\varepsilon)\|_H^2 )+ \right.
$$
$$\left.+( \hat{u}_n',\bar{u}_n-\hat{u}_n )_{L^2(\varepsilon,t;H)\times L^2(\varepsilon,t;H)}-\langle u',\bar{u}_n-\hat{u}_n \rangle_{L^q(\varepsilon,t;V^*)\times L^p(\varepsilon,t;V)}\right) = 0,
$$
and the proof concludes exactly the same as the proof of Theorem \ref{thm:improved}. 
\end{proof}


\begin{thebibliography}{99}

\bibitem{Aubin1984}
        J.P. Aubin and A. Cellina,
     Differential Inclusions. Set-Valued Maps and Viability Theory,
     Springer-Verlag, Berlin, New York, Tokyo (1984)

\bibitem{Ball1997}
     J.M. Ball,
     Continuity properties and global attractors of generalized semiflows and the Navier-Stokes equations,
     Nonlinear Science, \textbf{7}, 475--502 (1997)

\bibitem{Berkovits1996}
     J. Berkovits and V. Mustonen,
     Monotonicity methods for nonlinear evolution equations,
     Nonlinear Anal. Theory Methods Appl., \textbf{27}, 1397--1405 (1996)

\bibitem{Carl2006}
   S. Carl,
   Existence and comparison results for noncoercive and nonmonotone multivalued elliptic problems,
    Nonlinear Anal. Theory Methods Appl., \textbf{65}, 1532--1546 (2006)

\bibitem{Carstensen1999}
  C. Carstensen and J. Gwinner,
    A theory of discretization for nonlinear evolution inequalities applied to parabolic Signorini problems,
     Ann. Mat. Pura Appl., \textbf{177}, 363--394 (1999)

\bibitem{Chang1981}
   K.C. Chang,
    Variational methods for nondifferentiable functionals and their applications to partial differential equations,
     J. Math. Anal. Appl., \textbf{80}, 102--129 (1981)

\bibitem{Clarke1993}
   F.H. Clarke,
   Optimization and Nonsmooth Analysis,
   Wiley, New York (1993)

\bibitem{DMP2003Appl}
      Z. Denkowski, S. Mig\'{o}rski and N.S. Papageorgiou,
     An Introduction to Nonlinear Analysis, Applications,
      Kluwer (2003)


\bibitem{Emmrich2009}
        E. Emmrich,
    Variable time-step $\theta$-scheme for nonlinear evolution equations governed by a monotone operator,
      Calcolo, \textbf{46}, 187--210 (2009)


\bibitem{Kalita2012}
         P. Kalita,
     Convergence of Rothe scheme for hemivariational inequalities of parabolic type,
     Int. J. Numer. Anal. Mod., \textbf{10}, 445--465 (2013)


\bibitem{Kalita2012JMAA}
    P. Kalita,
    Regularity and Rothe method error estimates for parabolic hemivariational inequality,
      J. Math. Anal. Appl., \textbf{389}, 618--631 (2012)

\bibitem{Kloeden2009}
   P.E. Kloeden and J. Valero,
     Attractors of setvalued partial differential equations under discretization,
      IMA J. Numer. Anal., \textbf{29}, 690--711 (2009)

\bibitem{Liu2000}
   Z. Liu,
    A class of parabolic hemivariational inequalities,
     Appl. Math. Mech., \textbf{21}, 1045--1052 (2000)

\bibitem{Migorski2004}
   \newblock S. Mig\'{o}rski and A. Ochal,
     Boundary hemivariational inequality of parabolic type,
    Nonlinear Anal. Theory Methods Appl., \textbf{57}, 579--596 (2004)

\bibitem{Migorski2009}
    S. Mig\'{o}rski and A. Ochal,
    Quasistatic hemivariational inequality via vanishing acceleration approach,
     SIAM J. Math. Anal., \textbf{41}, 1415--1435 (2009)

\bibitem{Nochetto2000}
    R. H. Nochetto, G. Savare and C. Verdi,
     A posteriori error estimates for variable time-step discretizations of nonlinear evolution equations,
      Comm. Pure Appl. Math., \textbf{53}, 525--589 (2000)


\bibitem{Roubicek2005}
     T. Roubi\v{c}ek,
     Nonlinear Partial Differential Equations with Applications,
     Birkh\"{a}user Verlag, Basel, Boston, Berlin (2005)


\bibitem{Simon1987}
     J. Simon,
     Compact sets in the space $L^p(0,T;B)$,
     Ann. Mat. Pura Appl., \textbf{146}, 65--96 (1987)

\bibitem{Zeidler1990}
        E. Zeidler,
     Nonlinear Functional Analysis and its Applications, vol. II/B: Nonlinear Monotone Operators,
     Springer-Verlag, New York, Berlin, Heidelberg (1990)


\end{thebibliography}
\end{document}